\documentclass[12pt]{amsart}
\usepackage[margin=1in,includeheadfoot]{geometry}
\usepackage{amsmath}
\usepackage{slashed}
\usepackage{graphicx}

\usepackage[margin=1in,includeheadfoot]{geometry}
\usepackage{amsmath}
\usepackage{slashed}
\usepackage{graphicx}
\usepackage{mathrsfs}
\usepackage{txfonts}
\usepackage{amssymb,dsfont}
\usepackage{enumerate}
\usepackage{slashed}
\usepackage{fancyhdr}
\usepackage{aliascnt}
\usepackage{pdfsync}
\usepackage{hyperref}

\usepackage{color}
\newtheorem{thm}{Theorem}[section]

\newtheorem{lem}[thm]{Lemma}
\newtheorem{prop}[thm]{Proposition}
\theoremstyle{definition}
\newtheorem{defn}[thm]{Definition}
\theoremstyle{remark}
\newtheorem{rem}[thm]{Remark}
\numberwithin{equation}{section}

\newcounter{stepnum}

\def\bee{\begin{eqnarray}}
\def\beee{\begin{eqnarray*}}
\def\eee{\end{eqnarray}}
\def\eeee{\end{eqnarray*}}

\def\ba{\begin{array}}
\def\ea{\end{array}}

\def\S{\mathbb S}
\def\R{\mathbb R}

\def\s{\sigma}

\def\dist{{\rm dist\,}}
\begin{document}

\title[Toda-syatem] 
{The blow-up analysis of an affine Toda system corresponding to superconformal minimal surfaces in ${\mathbb S}^4$}

\author[Liu]{Lei Liu}
\address{School of Mathematics and Statistics \& Hubei Key Laboratory of
Mathematical Sciences, Central China Normal University, Wuhan 430079,
P. R. China}%
\email{llei1988@mail.ustc.edu.cn}

\author[Wang]{Guofang Wang}%
\address{Albert-Ludwigs-Universit\"{a}t Freiburg, Math Inst., Ernst-Zermelo-Str. 1,  Freiburg, D-79104 Germany}%
\email{guofang.wang@math.uni-freiburg.de}%

\thanks{}

\subjclass[2010]{}

\date{\today}
\begin{abstract}
	In this paper, we study the blow-up analysis of an affine Toda system corresponding  to minimal surfaces into ${\mathbb S}^4$ \cite{Ferus-Pinkall-Pedit-Sterling}.
	This system is an integrable system which is a natural   generalization of 
	sinh-Gordon equation \cite{FG}. 
	By exploring a refined blow-up analysis in the bubble domain, we prove that the blow-up values are multiple of $8\pi$, which generalizes  the previous results proved in \cite{Spruck, OS, Jost-Wang-Ye-Zhou, Jevnikar-Wei-Yang} for the sinh-Gordon equation.
	Let $(u_k^1,u_k^2, u^3_k)$  be a  sequence of solutions of
	\begin{align*}
	-\Delta u^1&=e^{u^1}-e^{u^3}, \\
	-\Delta u^2&=e^{u^2}-e^{u^3}, \\
	-\Delta u^3&=-\frac{1}{2}e^{u^1}-\frac{1}{2}e^{u^2}+ e^{u^3},\\
	u^1+u^2+2u^3&=0,
	\end{align*} in $B_1(0)$, which has a uniformly bounded energy in $B_1(0)$, a uniformly bounded oscillation on $\partial B_1(0)$ and blows up at an isolated blow-up point $\{0\}$, then the local masses
	$(\s_1,\s_2, \s_3) \not = 0$ satisfy
	\begin{align*}
	\begin{array}{rcl}
	\sigma_1&=&m_1(m_1+3)+m_2(m_2-1)\\
	\sigma_2&=& m_1(m_1-1)+m_2(m_2+3) \\
	\sigma_3 &=& m_1(m_1-1)+m_2(m_2-1) \end{array}  \, \qquad
	\hbox { for some }
	\begin{array} {l}  (m_1, m_2)\in {\mathbb Z}  \hbox { with }\\
	m_1, m_2= 0 \hbox{ or } 1 \hbox{ mod } 4,  \\ m_1, m_2 = 2   \hbox { or } 3\hbox { mod } 4. \end{array}
	\end{align*}
	Here the local mass is defined by $ \s_i:=\frac 1{2\pi}\lim_{\delta\to 0}\lim_{k\to\infty}\int_{B_\delta(0)}e^{u_k^i}dx.$
\end{abstract}
\maketitle
\section{Introduction}

The compactness of a solution space of a nonlinear equation or a system plays an important role in the study of the existence of solutions.
For most of interesting geometric partial differential equations it is a challenging problem to understand their solution space. One of important methods is the so-called blow-up analysis, which gives us better information,
when  the solution space  is lack of  the compactness.
This method goes back at least to
the work of Sacks-Uhlenbeck \cite{Sacks-Uhlenbeck} on harmonic maps. Since then, there has been a lot of work on the blow-up analysis for harmonic maps, and also for minimal surfaces,
the  equation of the constant mean curvature surface, pseudo-holomorphic curves,
Yang-Miles fields and the Yamabe equation \cite{Bahri-Coron}, \cite{Brezis-Coron}, \cite{Struwe}, which play an  important role in geometric analysis. In this paper we are interested in the blow-up  analysis for the following  Liouville type system,
\begin{equation}\label{eq_0}
-\Delta u_i =\sum_{j=1}^N a_{ij} e^{u_j}, \quad i=1,2, \cdots,  N,
\end{equation}
where $A=(a_{ij})$ is an $N\times N$ matrix, which has  usually   a  geometric meaning, for example the Cartan matrix of $SU(N+1)$. The Toda system is an important object in integrable system \cite{LS, FG}, appears in many physical models \cite{Dunne}
and
 relates to many geometric objects,  holomorphic curve, minimal surfaces,
harmonic maps and  flat connections. See for example  \cite{BJ,BW,Do, EW, JW0}.

When $N=1$ and $a_{11}=1$, it is the Liouville equation
\begin{equation}
\label{liou} -\Delta u =2 e^u,
\end{equation}
which plays a fundamental role in many two dimensional  physical models and  also in the problem of prescribed Gaussian curvature \cite{Chang-Yang}, \cite{Chang-Gursky-Yang}, \cite{Ding-Chen}.
Its blow-up analysis has been carried out by Brezis-Merle \cite{Brezis},  Li \cite{Li} and  Li-Shafrir \cite{Li-Shafrir}.
When  $N\ge 2$ and $A$ is the Cartan matrix
for $SU(N+1)$, the blow-up analysis for \eqref{eq_0} was initiated by Jost-Wang \cite{Jost-Wang} and continued by Jost-Lin-Wang \cite{Jost-Lin-Wang},  Lin-Wei-Ye
 \cite{Lin-Wei-Ye} and  many  mathematicians. There also have been a lot of work on the blow-up analysis for the Liouville system with conical singularities. Here we just mention the work of Chen-Li \cite{Chen-Li2}, Lin-Wei-Ye \cite{Lin-Wei-Ye} and the
recent  work of Lin-Wei-Yang \cite{Lin-Wei-Zhang}. When $A$ is the Cartan matrix
for a simple group, for example $S(N+1)$, \eqref{eq_0} is called  a Toda  system, or a 2-dimensional open (or finite) Toda system.
There are also work on the Toda  system  for other simple groups. Here we remark that such a system could be embedded in  the Toda system for $SU(N+1)$.
The solutions of an  (open) Toda system are closely related to holomorphic curves, harmonic maps and flat connections,  see \cite{BJ,BW,Do, EW}.

Other than these open Toda systems
there are another  type of  Liouville systems closely related to constant mean curvature surfaces,
superconformal  minimal surfaces and harmonic maps
rather than holomorphic curves,
 which are usually called
affine Toda systems or periodic Toda systems.
The corresponding Liouville systems  (or Toda systems) have also different features, though they are very close.
The simplest equation is the well-known equation, especially in mathematical physics,
the sinh-Gordon equation
\begin{equation}\label{eq_1}
-\Delta u= e^u-e^{-u}.
\end{equation}
The blow-up analysis of \eqref{eq_1} was initiated by Spruck in \cite{Spruck} and  was continued by   Ohtsuka-Suzuki in \cite{OS}, where the existence of solutions  were also studied.  A  more precise blow-up analysis was obtained by Jost-Wang-Ye-Zhou \cite{Jost-Wang-Ye-Zhou}, where they showed that  the local masses of blow-up must be multiples of $8\pi$. The proof uses the  connection of the solutions of \eqref{eq_1} with
constant mean curvature surfaces and harmonic maps. A possible drawback of this   geometric  proof is
that it may work  only for \eqref{eq_1}, not for the system with variable coefficients, namely,
\begin{equation}\label{eq_1a}
-\Delta u= h_1e^u-h_2 e^{-u},
\end{equation}
where $h_1$ and $h_2$ are two positive functions.\footnote{We believe that this geometric  method still works for the  blow-up analysis of \eqref{eq_1a}
	by considering
	constant mean curvature equation and harmonic maps with a small suitable perturbation. One can see a related consideration in \cite{Jost-Lin-Wang}.} The precise blow up analysis for \eqref{eq_1a} has been carried out recently by Jevnikar-Wei-Yang \cite{Jevnikar-Wei-Yang} by using a powerful  analytic method.  A similar blow-up analysis for the Tzitzeica equation was carried out in \cite{Jevnikar-Yang}.
In this paper we want to generalize their analysis to the following system
\begin{align}
-\Delta w&=e^{w}-\frac{1}{2}e^{-w+\eta}-\frac{1}{2}e^{-w-\eta},\label{equ:01}\\
-\Delta \eta&=\frac{1}{2}e^{-w+\eta}-\frac{1}{2}e^{-w-\eta}.\label{equ:02}
\end{align}
It is clear that the sinh-Gordon equation \eqref{eq_1} is a special case of \eqref{equ:01}-\eqref{equ:02} with $\eta=0$.
This system is related  to  minimal surfaces into ${\mathbb S}^4$  without superminimal points   \cite{Ferus-Pinkall-Pedit-Sterling}. It was appeared  already in the work on integrable system by   Fordy-Gibbons  in \cite{FG}.
For the geometric formulation, we refer to the classical paper by Ferus-Pinkall-Pedit-Sterling
\cite{Ferus-Pinkall-Pedit-Sterling}. For System \eqref{equ:01}-\eqref{equ:02}
it is convenient to consider the following equivalent system:
Let $$u^1=-w+\eta,\ u^2=-w-\eta,\ u^3=w.$$
System \eqref{equ:01}-\eqref{equ:02} is now equivalent to the following system
\begin{align}
-\Delta u^1&=e^{u^1}-e^{u^3}, \label{eq:03}\\
-\Delta u^2&=e^{u^2}-e^{u^3}, \label{eq:04}\\
-\Delta u^3&=-\frac{1}{2}e^{u^1}-\frac{1}{2}e^{u^2}+ e^{u^3},\label{eq:05}\\
u^1+u^2+2u^3&=0,\label{eq:06}
\end{align} in $B_1(0)$.

Our main result in this paper is the following

\begin{thm}
	\label{thm:main}
	Let $u^k=(u_k^1, u_k^2, u_k^3)$ be a sequence of solutions of System \eqref{eq:03}-\eqref{eq:06}
	satisfying
	\begin{align}\label{eq:13}
	\sum_{i=1}^3 \int_{B_1(0)}e^{u^k_i}(x)dx \le C, \quad
	\sum_{i=1}^3 osc_{\partial B_1(0)}u_k^i\leq C
	\end{align}
	and  with $0$ being its only blow-up point in $B_1(0)$, i.e.
	\begin{align}\label{eq:14}
	\max_{K\subset\subset \overline{B_1(0)}\setminus \{0\}}\big\{u_k^1, u_k^2, u_k^3\big\}\leq C(K),\ \quad  \max_{\overline{B_1(0)}}\big\{
	u_k^1, u_k^2, u_k^3
	\big\}\to+\infty.
	\end{align}
	Denote
	\[\sigma_i :=\frac 1{2\pi}\lim_{\delta\to 0}\lim_{k\to\infty}\int_{B_\delta(0)}e^{u_k^i}dx.\]
	Then
	\[ (\s_1,\s_2, \s_3)\in \mathbf{V},
	\]
	where $\mathbf{V}$ is defined
	\begin{align}\label{equation:14}
	\mathbf{V}:=&\big\{(\sigma_1,\sigma_2,\sigma_3)\ | \, \sigma_i =4n_i, n_i\in {\mathbb N}\cup \{0\}, i=1,2,3;
	\notag \\
	&\ (\sigma_1-\sigma_3)^2+(\sigma_2-\sigma_3)^2=4(\sigma_1+\sigma_2+2\sigma_3)\big\}\backslash\{0,0,0\}.
	\end{align}
\end{thm}

\

One can easily show that
\begin{align}\label{equation:14_a}
\mathbf{V}=&\left\{
\left(\begin{array}{c}\sigma_1\\ \sigma_2\\ \sigma_3\end{array}\right)
\in{\mathbb N^3_0}\, \bigg|\,
\begin{array}{rcl}
\sigma_1&=&m_1(m_1+3)+m_2(m_2-1)\\
\sigma_2&=& m_1(m_1-1)+m_2(m_2+3) \\
\sigma_3 &=& m_1(m_1-1)+m_2(m_2-1) \end{array}  \, \& \,
\begin{array} {l}  (m_1, m_2)\in {\mathbb Z} \\
m_1, m_2= 0 \hbox{ or } 1 \hbox{ mod } 4, \hbox { or } \\ m_1, m_2 = 2   \hbox { or } 3\hbox { mod } 4 \end{array}
\right\}
\backslash
\left(\begin{array}{c} 0\\0\\0\end{array}\right).
\end{align}
Here ${\mathbb N_0}={\mathbb N}\cup\{0\}$.
The first  5 possibilities of $(\s_1, \s_2, \s_3)$ are
\[ (0,0, 4), (4,0,0), (4,4, 0), (4,4,12) , (12,12,4).
\]
Since the roles of $\s_1$ and $\s_2$ are symmetric, we have not listed the obvious possibilities by changing $\s_1$ and $\s_2$.  When $u_1=u_2$, it is easy to see that  System \eqref{eq:03}--\eqref{eq:06} is reduced to the sinh-Gordon equation. The above result generalizes the previous results given in
\cite{Spruck, OS, Jost-Wang-Ye-Zhou, Jevnikar-Wei-Yang}  for sinh-Gordon equation.

The result holds also   for
System \eqref{eq:03}--\eqref{eq:05} with variable coefficients.
We remark that like the sinh-Gordon equation one may expect all cases listed in  $\mathbf{V}$ could occur. For the related work we refer to the work of Esposito-Wei   \cite{E_W} and Grossi-Pistoia \cite{G_P} and reference therein. We leave this problem to the interested reader.
With  Theorem \ref{thm:main} one can obtain the existence of solutions for system \eqref{equ:01}-\eqref{equ:02} under suitable conditions.  See for example \cite{DJLW}, \cite{Jost-Lin-Wang} and  the work of Malchiodi for an approach to
the existence \cite{Malchiodi}.

This paper follows closely the arguments given in \cite{Jevnikar-Wei-Yang}.
We do blow-up analysis for our system and obtain blow-up limits which are solution of
either the Liouville equation \eqref{liou} or the following system
\begin{align}\label{System2}
\begin{cases}
-\Delta v_1 &=   e^{v_1}-e^{v_2} \\
-\Delta v_2 &=   -\frac{1}{2} e^{v_1}+e^{v_2}.
\end{cases}\end{align}
 Except this  complexity, there is one big difference. In order to
 describe this difference,
let us   first  recall  the key  ideas in \cite{Jevnikar-Wei-Yang} for the  sinh-Gordon equation. By using a selection process for describing the blow-up situations in the framework of prescribed curvature problems (see for example \cite{Chen-Lin,Li}), the key point in \cite{Jevnikar-Wei-Yang} is to show that in each bubbling disk
 at least local mass  ($\sigma_1$ or $\sigma_2$) is a multiple of $4$.
 More precisely, let $B(0,\lambda_k)$ be a bubbling disk. By using a standard blow-up analysis, one can easily get that there exists a sequence of numbers $r^1_k\to 0$ with $r^1_k=o(1)\lambda_k$, such that $$(\sigma_1(r^1_k),\sigma_2(r^1_k))=(4,0)+o(1), \quad \hbox{ or }\ (0,4)+o(1),$$ which roughly means that in $B_{r^1_k}(0)$ one component of $u_k$ converges to a solution of the
Liouville equation while the other  converges to $-\infty$. Here
$\sigma_i(r) = \frac 1 {2\pi} \int_{B_r} e^{u_i}$ is the so-called local mass (or local energy) of $u^i$. As $r$ increases from $r^1_k$ to $\lambda_k$, there are two possibilities: either

(1) $(\sigma_1(\lambda_k),\sigma_2(\lambda_k))=(4,0)+o(1)$ or $(0,4)+o(1)$, or

(2) there exists $s_k\to 0$ with ${s_k}\slash {r^1_k}\to +\infty$ and  $s_k=o(1)\lambda_k$, such that one component of the solutions has slow decay (for the definition, see Definition \ref{def-fast-decay} below).
One only needs to care about case (2). In this case, the authors in \cite{Jevnikar-Wei-Yang} managed to show that as $r$ increases across  $s_k$, one of
$\sigma_i $ ($i=1,2$)  almost does not change
 while the other increases at least a  positive quantity.
More precisely they proved that there exists $r^2_k\to 0$ with $r^2_k=o(1)\lambda_k$ and  ${r^2_k}\slash{s_k}\to +\infty$ such that $$\sigma_1(r^2_k)-\sigma_1(r^1_k)=o(1),\ \ \sigma_2(r^2_k)-\sigma_2(r^1_k)\geq \delta>0,\  \hbox{ for some small } \delta>0.$$
Then  a local  Pohozaev identity
\begin{equation*}\label{equation:pohozaev-sinh-Gordon}
(\sigma_1(r^2_k)-\sigma_2(r^2_k))^2=4(\sigma_1(r^2_k)+\sigma_2(r^2_k) )+o(1)
\end{equation*}
for the sinh-Gordon equation implies that
$$\sigma_1(r^2_k)-\sigma_1(r^1_k)=o(1),\ \ \sigma_2(r^2_k)-\sigma_2(r^1_k)=4n+o(1),\ \ n\in\mathds{N}.$$ This means the local energies  increase always a multiple of $4$, whenever   $r$ increases across such an $s_k$.
 Then the results follow from a careful iteration argument.

For our system \eqref{eq:03}--\eqref{eq:05} we also have  a similar
 local  Pohozaev identity
\begin{align}
\big(\sigma_k^1(r_k)-\sigma_k^3(r_k)\big)^2+ \big(\sigma_k^2(r_k)-\sigma_k^3(r_k)\big)^2\notag  =4\big(\sigma_k^1(r_k)+\sigma_k^2(r_k)+2\sigma_k^3(r_k)\big)+o(1).
 \end{align} See Section \ref{sec:basic-lemmas} below. This identity  plays also a crucial role in our proof. However, alone with  this identity is not enough
 for our system, since we have 3 local energies. It could happen that
  when $r$ increases,
 one of local energies keeps almost no change, other two local  energies increase
  at least a  positive quantity.
  One  could not use the local Pohozaev identity  to conclude that these two
   energies must increase a multiple of $4$.  In order to deal with this problem,
   we manage to do another  blow-up  near $s_k$ and obtain a ``singular bubble",
   which is  either a solution of the Liouville equation  \eqref{liou} or
   the system \eqref{System2}, but with a singular source at the origin. See \eqref{eq_addf1} and \eqref{eq_add_f2}. The classification result in the recent work of Lin-Yang-Zhong \cite{Lin-Yang-Zhong} tells us that the corresponding
   local energies are a multiple of $4$. Moreover, we
   show that between the previous blow-up and  the singular bubble there is no energy loss. In order to show this we crucially  use an oscillation estimate Lemma \ref{lem:02} and the local Pohozaev identity. This is the main difference to the paper of Jevniker-Wei-Yang  \cite{Jevnikar-Wei-Yang}.
   This proof's ideas come from the study of the harmonic maps, where one proves
   the so-called energy identity, see for example the work of Ding-Tian \cite{Ding-T}.

   We want to emphasize that in the above blow-up analysis the classification of all entire solution of
   the blow-up limits plays a crucial role.
   The classification for the Liouville equation \eqref{liou} was given by Chen-Li \cite{Chen-Li},
   for \eqref{eq_0} for $SU(N+1)$  by Jost-Wang \cite{Jost-Wang}. As mentioned above
   we need also
   the classification of entire solutions to the Liouville equation \eqref{liou} and system \eqref{System2} with  a singular source. For such a result see the work of  Lin-Wei-Ye \cite{Lin-Wei-Ye} and Lin-Yang-Zhong \cite{Lin-Yang-Zhong} or Appendix  below.

   Our methods work at least also for the affine Toda system for $SU(4)$, which includes
   the Tzitzeica equation as a special case. See Section \ref{sec_last} below. The blow-up analysis of the Tzitzeica equation \cite{Tzitzeca-2} was carried out by
   Jevniker-Yang \cite{Jevnikar-Yang} recently.

\

The rest of the paper is organized as follows. In Section \ref{sec:basic-lemmas}
	we establish a selection process  for finite number of bubbling areas, the oscillation estimate outside the blow-up set and the local Pohozaev identities
	corresponding to the blow-up analysis of system \eqref{equ:03}-\eqref{equ:06},
	In Section \ref{sec:local-blowup-behavior}, we prove a local blow-up behavior Theorem \ref{thm:01} where we need to explore more careful blow-up analysis in the bubbling areas. With the help of local  Theorem \ref{thm:01}, by using a standard argument of combining the blow-up areas and a global Pohozaev identity, we give the proof of our main Theorem \ref{thm:main}  in Section \ref{sec:main-theorem}. In Section  \ref{sec_last}, we discuss another systems, affine Toda system for $SU(4)$. The blow-up analysis holds for this system.
	 In Section \ref{sec:appendix}, we recall  two classification theorems which are used in our proof.

\section{Some basic lemmas}\label{sec:basic-lemmas}

\

In this section, we prove the selection of bubbling areas, a crucial oscillation estimate and the local  Pohozaev identity, which play important roles in our later proof.

Let us repeat the system and the assumptions we will use in this paper
\begin{align}
  -\Delta u^1&=e^{u^1}-e^{u^3}, \label{equ:03}\\
  -\Delta u^2&=e^{u^2}-e^{u^3}, \label{equ:04}\\
  -\Delta u^3&=-\frac{1}{2}e^{u^1}-\frac{1}{2}e^{u^2}+ e^{u^3},\label{equ:05}\\
  u^1+u^2+2u^3&=0,\label{equ:06}
\end{align} in $B_1(0)$.
We consider a sequence of solutions $u_k=(u^1_k,u^2_k,u^3_k)$ of \eqref{equ:03}-\eqref{equ:06} with a uniformly bounded energy
\begin{align}\label{equ:07}
 \sum_{i=1}^3\int_{B_1(0)}e^{u^i_k(x)}dx<C
\end{align} and  with $0$ being its only blow-up point in $B_1(0)$, i.e.
\begin{align}\label{equ:08}
 \max_{i=1,2,3}\max_{K\subset\subset \overline{B_1(0)}\setminus \{0\}}u^i_k\leq C(K),\ \quad  \max_{i=1,2,3}\max_{\overline{B_1(0)}}u^i_k\to+\infty,\ \quad \max_{i=1,2,3}osc_{\partial B_1}u^i_k\leq C.
\end{align}

\

We first establish a lemma for the selection of bubbling areas.

\begin{lem}\label{lem:01}
Let $u_k=(u^1_k,u^2_k,u^3_k)$ be a sequence of \eqref{equ:03}-\eqref{equ:06} with \eqref{equ:07} and \eqref{equ:08}. Then there exists a sequence of finite points $\Sigma_k:=\{x^1_k,...,x^m_k\}$ and a sequence of positive numbers $\lambda^1_k,...,\lambda^m_k$ such that
\begin{itemize}
  \item[(1)] $x^j_k\to 0$  and $ \lambda^j_k\to 0$ as $k\to +\infty,\ \ \lambda^j_k\leq \frac{1}{2}\dist(x^j_k,\Sigma_k\setminus \{x^j_k\}), \ j=1,...,m$;

  \

  \item[(2)] $B_{\lambda^j_k}(x^j_k)\cap B_{\lambda^l_k}(x^l_k)=\emptyset$, for $1\leq j,l\leq m,\ j\neq l$;

  \

  \item[(3)] $\max_{i=1,2,3}u^i_k(x^j_k)=\max_{i=1,2,3}\max_{B_{\lambda^j_k}(x^j_k)}u^i_k\to +\infty$ as $k\to +\infty,\ j=1,...,m;$

  \

  \item[(4)] Denote $\epsilon^j_k:=\exp(-\frac{1}{2}\max_{i=1,2,3}u^i_k(x^j_k))$. Then $\frac{\lambda^j_k}{\epsilon^j_k}\to +\infty$  as $k\to +\infty,\ j=1,...,m;$

      \

  \item[(5)] In each $B_{\lambda^j_k}(x^j_k)$, we define the scaled functions $$v^i_k(x):=u^i_k(x^j_k+\epsilon^j_kx)+2\log\epsilon^j_k,\ \ i=1,2,3.$$ Then one of the following alternatives holds:
      \begin{itemize}
        \item[(5-1)] Two components of them $v^1_k$ and $v^3_k$ (or $v^2_k$ and $v^3_k$) converge to a solution of Toda system \eqref{equ:10} in $C^2_{loc}(\R^2)$, while the left one converges to $-\infty$ over all compact subsets of $\R^2$;

            \

        \item[(5-2)] Two components of them $v^1_k$ and $v^2_k$ converge to solutions of Liouville equation in $C^2_{loc}(\R^2)$, while the left one converges to $-\infty$ over all compact subsets of $\R^2$;

            \

        \item[(5-3)] One component  converges to a solution of Liouville equation in $C^2_{loc}(\R^2)$, while the left ones converge to $-\infty$ over all compact subsets of $\R^2$;
      \end{itemize}

      \

  \item[(6)] There exists a constant $C$ independent of $k$ such that $$\max_{i=1,2,3}u^i_k(x)+2\log \dist(x,\Sigma_k)\leq C,\ \ \forall \ x\in B_1. $$
\end{itemize}
\end{lem}
\begin{proof} The proof is becoming more or less standard.
We divide the proof into three steps.

\

\noindent
\textbf{Step 1:} Construction of  the first point $x^1_k$ and $\lambda^1_k$.

\

Let $x^1_k$ be the maximal point such that $$\max_{i=1,2,3}u^i_k(x^1_k)=\max_{i=1,2,3}\max_{x\in \overline{B_1}}u^i_k(x).$$ By assumption, it is clearly that $x^1_k\to 0$ and $\max_{i=1,2,3}u^i_k(x^1_k)\to +\infty$ as $k\to +\infty$. According to \eqref{equ:06}, we know $$u^1_k(x^1_k)+u^2_k(x^1_k)+2u^3_k(x^1_k)=0,$$ which implies that one of the following cases happens:
\begin{itemize}
  \item[(I-1)] $(u^1_k(x^1_k),u^2_k(x^1_k),u^3_k(x^1_k))\to (+\infty,-\infty,+\infty)$ and $|u^1_k(x^1_k)-u^3_k(x^1_k)|\leq C$;
  \item[(I-2)] $(u^1_k(x^1_k),u^2_k(x^1_k),u^3_k(x^1_k))\to (+\infty,-\infty,+\infty)$ and $|u^1_k(x^1_k)-u^3_k(x^1_k)|\to +\infty $;
  \item[(I-3)] $(u^1_k(x^1_k),u^2_k(x^1_k),u^3_k(x^1_k))\to (+\infty,+\infty, -\infty)$ and $|u^1_k(x^1_k)-u^2_k(x^1_k)|\leq C$;
  \item[(I-4)] $(u^1_k(x^1_k),u^2_k(x^1_k),u^3_k(x^1_k))\to (+\infty,+\infty, -\infty)$ and $|u^1_k(x^1_k)-u^2_k(x^1_k)|\to +\infty$;
  \item[(I-5)] There exists one component $u^i_k(x^1_k)\to +\infty$, $i\in\{1,2,3\}$ while the other ones have uniformly upper bound, i.e. $u^l_k(x^1_k)\leq C,\ l\neq i$.
\end{itemize}
Recall that we do not list the cases which just change $1$ and $2$.

Now, we will show  that case (I-1)   leads to  conclusion (5-1),  (I-3)  to (5-2), while  cases (I-2),\ (I-4),\ (I-5)  lead to  conclusion (5-3).
We only  consider  case (I-1), since the other cases are similar or easier. Let $$v^i_k(x):=u^i_k(x^1_k+\epsilon^1_kx)+2\log\epsilon^1_k,\ \ i=1,2,3,$$ where $$\epsilon^1_k:=\exp(-\frac{1}{2}\max_{i=1,2,3}u^i_k(x^1_k))=\exp(-\frac{1}{2}\max_{i=1,3}u^i_k(x^1_k)).$$ Then it is easy to see that $$\max_{i=1,2,3}v^i_k(x)\leq 0,\ \ -C\leq v^1_k(0)\leq 0,\ \ -C\leq v^3_k(0)\leq 0,\ \ v^2_k(0)\to -\infty.$$ By the standard theory of
the elliptic equation we know that $v^2_k(x)\to -\infty$ in $L^\infty_{loc}(\R^2)$ and $(v^1_k, v^3_k)\to (v^1,v^3)$ in $C^2_{loc}(\R^2)$, where $(v^1, v^3)$ satisfies
\begin{align}\label{equ:10}
\begin{cases}
  -\Delta v^1&=e^{v^1}-e^{v^3},\\
  -\Delta v^3&=-\frac{1}{2}e^{v^1}+e^{v^3},
\end{cases} \quad \hbox{ in } \R^2.
\end{align} By the classification Theorem \ref{thm:classification-1} below, we have
\begin{equation}\label{equ:11}
\int_{\R^2}e^{v^1}dx=32\pi,\ \ \int_{\R^2}e^{v^3}dx=24\pi \end{equation} and
\begin{equation}\label{equ:12}
v^1(x)=-4\log |x|+O(1),\ \ v^3(x)=-4\log |x|+O(1)\ as \ |x|\to\infty.\end{equation}
Then it is not hard to take a sequence $R_k\to +\infty$ such that $R_k\epsilon^1_k\to 0$ and $$\max_{i=1,2,3}v^i_k(x)+2\log |x|\leq C,\ \ |x|\leq R_k.$$ Let $\lambda_k^1:=\frac{1}{2}R_k\epsilon^1_k$. Then ${\lambda_k^1}\slash {\epsilon^1_k}\to +\infty$ and $$\max_{i=1,2,3}u^i_k(x)+2\log |x-x^1_k|\leq C,\ \ |x-x^1_k|\leq 2\lambda^1_k.$$

\

\noindent
\textbf{Step 2:} Construction of  the second point $x^2_k$ and $\lambda^2_k$.

\

Now we consider the function $$\max_{i=1,2,3}u^i_k(x)+2\log |x-x^1_k|.$$ If it is uniformly bounded in $\overline{B_1}$, then the lemma holds for $m=1$. Otherwise, there exists $y_k\in \overline{B_1}$ such that $$\max_{i=1,2,3}u^i_k(y_k)+2\log |y_k-x^1_k|=\max_{i=1,2,3}\max_{x\in \overline{B_1}}u^i_k(x)+2\log |x-x^1_k|\to +\infty.$$ Let $d_k:=\frac{1}{2}|x^1_k-y_k|$ and $$w^i_k(x):=u^i_k(x)+2\log (d_k-|x-y_k|),\ \ i=1,2,3.$$ Let $p_k$ be a maximal point such that $$\max_{i=1,2,3}w^i_k(p_k)=\max_{i=1,2,3}\max_{x\in \overline{B_{d_k}(y_k)}}w^i_k(x).$$ Then $$\max_{i=1,2,3}w^i_k(p_k)\to +\infty,$$  since $\max_{i=1,2,3}w^i_k(y_k)\to +\infty.$ In view of  $$\sum_{i=1}^3w^i_k(p_k)=6\log (d_k-|p_k-y_k|)\leq 0,$$ similar to
 \textbf{Step 1}, we have that one of the following cases happens:
\begin{itemize}
  \item[(II-1)] $(w^1_k(p_k),w^2_k(p_k),w^3_k(p_k))\to (+\infty,-\infty,+\infty)$ and $|w^1_k(p_k)-w^3_k(p_k)|\leq C$.
  \item[(II-2)] $(w^1_k(p_k),w^2_k(p_k),w^3_k(p_k))\to (+\infty,-\infty,+\infty)$ and $|w^1_k(p_k)-w^3_k(p_k)|\to +\infty $.
  \item[(II-3)] $(w^1_k(p_k),w^2_k(p_k),w^3_k(p_k))\to (+\infty,+\infty, -\infty)$ and $|w^1_k(p_k)-w^2_k(p_k)|\leq C$.
  \item[(II-4)] $(w^1_k(p_k),w^2_k(p_k),w^3_k(p_k))\to (+\infty,+\infty, -\infty)$ and $|w^1_k(p_k)-w^2_k(p_k)|\to +\infty$.
  \item[(II-5)] There exists one component $w^i_k(p_k)\to +\infty$, $i\in\{1,2,3\}$ while the other ones have uniformly upper bound, i.e. $w^l_k(p_k)\leq C,\ l\neq i$.
\end{itemize}

Next, we will show that the existence of $x^2_k$ and $\lambda_k^2$ for case (II-1), which satisfy the conclusions of the lemma. The other cases are similar or easier.
Let $$l_k:=\frac{1}{2}(d_k-|p_k-y_k|)\  \hbox{ and }  \ \epsilon_k:= \exp(-\frac{1}{2}\max_{i=1,2,3}u^i_k(p_k)).$$
From the fact that $\max_{i=1,2,3}w^i_k(p_k)\to +\infty$, we have $$\epsilon_k\to 0\quad  \hbox{ and }\ \quad  \frac{l_k}{\epsilon_k}\to +\infty.$$  We know  that $\forall x\in B_{l_k}(p_k)$, there hold $$\max_{i=1,2,3}u^i_k(x)+2\log (d_k-|x-y_k|)\leq \max_{i=1,2,3}w^i_k(p_k)= \max_{i=1,2,3}u^i_k(p_k)+2\log (2l_k)$$ and $$d_k-|x-y_k|\geq d_k-|p_k-y_k|-|x-p_k|\geq l_k.$$ It follows that  $$\max_{i=1,2,3}u^i_k(x)\leq \max_{i=1,2,3}u^i_k(p_k)+2\log 2,\ \ \forall x\in B_{l_k}(p_k) .$$ Now, let $$v^i_k(x):=u^i_k(p_k+\epsilon_kx)+2\log\epsilon_k,\ \ |x|\leq\frac{l_k}{\epsilon_k},\ \  i=1,2,3.$$ For  case (II-1), we can easily see that $$\max_{i=1,2,3}v^i_k(x)\leq 2\log 2,\ \ -C\leq v^1_k(0)\leq 0,\ \ -C\leq v^3_k(0)\leq 0,\ \ v^2_k(0)\to -\infty. $$ By the standard theory of Laplacian operator, we know that $v^2_k(x)\to -\infty$ in $L^\infty_{loc}(\R^2)$ and $(v^1_k, v^3_k)\to (v^1,v^3)$ in $C^2_{loc}(\R^2)$ where $(v^1, v^3)$ satisfies \eqref{equ:10}, \eqref{equ:11} and \eqref{equ:12}. By \eqref{equ:12}, we assume $$v^1(\overline{x}^1)=\max_{x\in\R^2}v^1(x),\ \ v^3(\overline{x}^3)=\max_{x\in\R^2}v^3(x),\ \quad  \hbox{  for
	 some }\ \overline{x}^1,\ \overline{x}^3\in \R^2.$$ Moreover, we can choose a sequence of $R_k\to +\infty$ such that $R_k=o(1){l_k}\slash {\epsilon_k}$ and $$v^i_k(x)+2\log |x|\leq C,\ \ |x|\leq R_k,\ i=1,2,3$$ and $$\int_{R_k}e^{v^1_k(x)}dx=32\pi+o(1),\ \ \int_{R_k}e^{v^3_k(x)}dx=24\pi+o(1).$$ Let $q_k\in B_{\frac{1}{2}R_k}(0)$ such that $$\max\{v^1_k(\overline{x}^1+q_k), v^3_k(\overline{x}^3+q_k) \}=\max\{\max_{x\in \overline{B_{\frac{1}{2}R_k}(0)}}v^1_k(\overline{x}^1+x),\max_{x\in \overline{B_{\frac{1}{2}R_k}(0)}} v^3_k(\overline{x}^3+q_k) \}.$$ Now set
\begin{align*}
x^2_k:=
\begin{cases}
  p_k+\epsilon_k(\overline{x}^1+q_k), & \mbox{if } v^1(\overline{x}^1)\geq v^3(\overline{x}^3), \\
  p_k+\epsilon_k(\overline{x}^3+q_k), & \mbox{otherwise},
\end{cases}
\end{align*} and $$\lambda_k^2:=\frac{1}{4}R_k\epsilon_k.$$ One can check that $(x^2_k,\lambda_k^2)$ satisfies all statements of the lemma.

\

\noindent \textbf{Step 3:}
By above two steps, we have defined the selection process. Continuously, we consider the function $$\max_{i=1,2,3}u^i_k(x)+2\log \dist(x,\{x^1_k,x^2_k\}).$$ If it is uniformly bounded, then we stop and conclude the lemma for $m=2$. Otherwise, using the same argument, we get $x^3_k$ and $\lambda_k^3$. Since each bubble area $B_{\lambda_k^j}(x_k^j)$ contributes a positive energy, the   above process must stop after finite steps due to the energy bound \eqref{equ:07}. We proved the lemma.
\end{proof}

\

Next we prove an oscillation estimate.
\begin{lem}\label{lem:02}
Let $u^i_k,\ i=1,2,3$ be the solution of \eqref{equ:03}-\eqref{equ:06} satisfying \eqref{equ:07} and \eqref{equ:08}. Then $$osc_{B_{\frac{1}{2}d_k(x)}}u^i_k\leq C,\qquad  \forall x\in D\setminus \Sigma_k,\  i=1,2,3,$$ where $d_k(x):=\dist(x, \Sigma_k)$, $C$ is a constant independent of $x,\ k$ and $$osc_\Omega u:=\sup_{x,y\in\Omega}(u(x)-u(y)).$$
\end{lem}
\begin{proof}
Since $$\max_{i=1,2,3}\max_{K\subset \overline{B_1}\setminus \{0\}}u^i_K\leq C(K),\ \ \max_{i=1,2,3}osc_{\partial B_1}u^i_k\leq C,$$ by a standard argument using Green's representation formula, we know $$\max_{i=1,2,3}osc_{ \overline{B_1}\setminus B_\delta}u^i_k\leq C(\delta).$$ Hence, we just need to prove the lemma in $ B_{\frac{1}{20}}$. Let $v^i_k$ be the solution of
\begin{align*}
\begin{cases}
  -\Delta v^i_k=0,\ \ &in\ \ B_1(0),\\
  v^i_k=u^i_k,\ \ &on\ \ \partial B_1(0).
\end{cases}
\end{align*} It is clear that $$osc_{\overline{B_1}}v^i_k\leq osc_{\partial B_1}v^i_k\leq C.$$

Let $w^i_k:=u^i_k-v^i_k$ and let $$G(x,y)=-\frac{1}{2\pi}\log |x-y|+H(x,y)$$ be  Green's function on $B_1$ with respect to the Dirichlet boundary, where $H(x,y)$ is a smooth harmonic function. Then $$w^i_k(x)=\int_{B_1}G(x,y)(-\Delta w^i_k(y))dy.$$
Let $x_0\in B_{\frac{1}{20}}$ and $r_k:=\dist(x_0,\Sigma_k)$. For any $x_1,x_2\in B_{\frac{1}{2}r_k}(x_0)$, then
\begin{align*}
w^i_k(x_1)-w^i_k(x_2)&=\int_{B_1}(G(x_1,y)-G(x_2,y))(-\Delta w^i_k(y))dy\\
&=\frac{1}{2\pi}\int_{B_1}\log\frac{|x_1-y|}{|x_2-y|}(\Delta w^i_k(y))dy+O(1).
\end{align*}
We divide the integral into two parts, i.e. $B_1=B_{\frac{3}{4}r_k}(x_0)\cup B_1\setminus B_{\frac{3}{4}r_k}(x_0)$. Noting that $$\bigg|\log\frac{|x_1-y|}{|x_2-y|}\bigg|\leq C,\ \ y\in B_1\setminus B_{\frac{3}{4}r_k}(x_0),$$ we have $$|w^i_k(x_1)-w^i_k(x_2)|\leq C(1+\|\Delta w^i_k\|_{L^1(B_1)})\leq C.$$
A direct computation yields
\begin{align*}
  \bigg|\int_{B_{\frac{3}{4}r_k}(x_0)}\log\frac{ |x_1-y|}{ |x_2-y|}(\Delta w^i_k(y))dy \,\bigg|&\leq C\int_{B_{\frac{3}{4}r_k}(x_0)}\bigg|\log\frac{ |x_1-y|}{ |x_2-y|}\bigg| \sum_{i=1}^3e^{u^i_k(y)}dy\\
  &=C\int_{B_{\frac{3}{4}}(0)}\bigg| \log\frac{ |x_1-x_0-r_ky|}{ |x_2-x_0-r_ky|}\bigg| \sum_{i=1}^3e^{u^i_k(x_0+r_ky)}r_k^2dy\\
  &=C\int_{B_{\frac{3}{4}}(0)}\bigg| \log\frac{ \big|\frac{x_1-x_0}{r_k}-y\big|}{ \big|\frac{x_2-x_0}{r_k}-y\big|}\bigg| \sum_{i=1}^3e^{u^i_k(x_0+r_ky)}r_k^2dy.
\end{align*} Since  $$\dist(x_0+r_ky,\Sigma_k)\geq \dist(x_0,\Sigma_k)-\dist(x_0,x_0+r_ky)\geq\frac{1}{4}r_k,\ \ \forall\ y\in B_{\frac{3}{4}},$$ by lemma \ref{lem:01}, we have $$u^i_k(x_0+r_ky)\leq C-2\log \dist(x_0+r_ky,\Sigma_k)\leq C-2\log r_k,\ \ \forall \ y\in B_{\frac{3}{4}}.$$ From  $$\big|\frac{x_1-x_0}{r_k}\big|\leq\frac{1}{2},\ \ \big|\frac{x_2-x_0}{r_k}\big|\leq\frac{1}{2},$$ it is easy to conclude that $$ \big|\int_{B_{\frac{3}{4}r_k}(x_0)}
\log\frac{ |x_1-y|}{ |x_2-y|}(\Delta w^i_k(y))dy \big|\leq C.$$ This implies $$|w^i_k(x_1)-w^i_k(x_2)|\leq C.$$ Then the conclusion of the lemma follows.
\end{proof}

\

Before giving the next lemma, we first need definitions of fast decay and slow decay, which were  used in \cite{Lin-Wei-Zhang}.
\begin{defn}\label{def-fast-decay}
 (i)
We say $u_k$ has fast decay on $\partial B_{r_k}(x_0)$ (resp.  $ B_{r_k}(x_0)\backslash B_{s_k}(x_0)$)
if $$u_k(x)+2\log |x|\leq -N_k,\ \ \forall\ x\in \partial B_{r_k}(x_0) \quad
(\hbox{resp. } \forall x \in   B_{r_k}(x_0)\backslash B_{s_k}(x_0) ),$$ for some $N_k\to +\infty$ as $k\to \infty$.

(ii)
We say $u_k$ has slow decay on $\partial B_{r_k}(x_0)$ if $$\sup_{x\in\partial B_{r_k}(x_0)}(u_k(x)+2\log |x|)\geq -C,$$ for some $C>0$ which is independent of $k$.
\end{defn}

Remark that in this paper we use many times the notation $a_k=o(1)\, b_k$,
 which
means certainly that $a_k/b_k \to 0$ as $k\to \infty.$
We also use the notation $a_k=o(1)^{-1}\, b_k $, which
means  that $a_k/b_k \to \infty $ as $k\to \infty.$ We will use also the notation $a_k=O(1)\, b_k$, which means certainly that $C^{-1}<a_k/b_k <C $ as $k\to \infty.$

\begin{lem}\label{lem:05}
  Let $\Sigma^1_k\subset \Sigma_k$ a subset of  $\Sigma_k$  with  $\Sigma^1_k\subset B_{r_k}(x_k)\subset B_1(0)$ and  $$\dist(\Sigma_k^1,\partial B_{r_k}(x_k))=o(1)\, \dist(\Sigma_k\setminus \Sigma_k^1, \partial B_{r_k}(x_k) ).$$ Then for any $s_k\geq 2r_k$ with $s_k=o(1) \, \dist(\Sigma_k\setminus \Sigma_k^1, \partial B_{r_k}(x_k) )$, we have:
  \begin{itemize}
    \item[(1)] For fixed $\ i\in\{1,2,3\}$, if $u^i_k$  has fast decay on $\partial B_{s_k}$, then for any $\beta_k\to+\infty$ with $\beta_ks_k=o(1)\,  \dist(\Sigma_k\backslash \,  \Sigma_k^1, \partial B_{r_k}(x_k) )$, there exists $\alpha_k\to +\infty$ with $\alpha_k=o(1)\, \beta_k$ such that  $u^i_k$ has fast decay in $B_{\alpha_ks_k}\setminus B_{s_k}$, i.e. $$u^i_k(x)+2\log |x|\leq -N_k,\ \ \forall\ s_k\leq |x-x_k|\leq \alpha_ks_k,$$ for some $N_k\to+\infty$, and $$\int_{B_{\alpha_ks_k}\setminus B_{s_k}}e^{u^i_k}dx=o(1).$$

    \item[(2)] For fixed $\ i\in\{1,2,3\}$, if $u^i_k$ has fast decay on $\partial B_{s_k}$ with $s_k=o(1)^{-1}
r_k$,
    then for any $\beta_k\to 0$ with $\beta_ks_k\geq 2r_k$, there exists $\alpha_k\to 0$ with $\alpha_k=o(1)^{-1} \beta_k$  such that  $u^i_k$ has fast decay in $B_{s_k}\setminus B_{\alpha_ks_k}$, i.e. $$u^i_k(x)+2\log |x|\leq -N_k,\ \ \forall\ \alpha_ks_k\leq |x-x_k|\leq s_k,$$ for some $N_k\to+\infty$, and $$\int_{B_{s_k}\setminus B_{\alpha_ks_k}}e^{u^i_k}dx=o(1).$$
    \item[(3)] For any $\beta_k\to+\infty$ with $\beta_ks_k=o(1)\, \dist(\Sigma_k\setminus \Sigma_k^1, \partial B_{r_k}(x_k) )$, there exists $\alpha_k\to +\infty$ with  $\alpha_k=o(1)\beta_k$ such that $u^1_k,\ u^2_k,\ u^3_k$ have fast decay on $\partial B_{\alpha_ks_k}$.
    \item[(4)] If ${s_k}=o(1)^{-1} {r_k}$, then for any $\beta_k\to 0$ with $\beta_k s_k\geq 2r_k$, there exists $\alpha_k\to 0$ with $\alpha_k = o(1)^{-1} \beta_k$ 
     such that $u^1_k,\ u^2_k,\ u^3_k$ have fast decay on $\partial B_{\alpha_k s_k}$.
  \end{itemize}
\end{lem}
\begin{proof}
Without loss of generality, we assume $x_k=0$.

\

\noindent\textbf{Step 1:} We prove statements  $(1)$ and $(2)$.

\

Since $u^i_k$ has fast decay on $\partial B_{s_k}$, i.e. $$u^i_k(x)+2\log |x|\leq -N_k,\ \ \forall\  |x|=s_k,$$ for some $N_k\to+\infty$, for any fixed $\Lambda>0$,  by Lemma \ref{lem:02}, there holds $$u^i_k(x)+2\log |x|\leq -N_k+C(\Lambda),\ \ \forall\ s_k\leq |x|\leq \Lambda s_k.$$ Then it is not hard to see that there exists $\alpha^1_k\to +\infty$ such that $\alpha^1_k=o(1)\beta_k$ and $$u^i_k(x)+2\log |x|\leq -\frac{1}{2}N_k,\ \ \forall\ s_k\leq |x|\leq \alpha_k^1 s_k.$$ Now, we can choose $\alpha_k\to +\infty$ such that $\alpha_k=o(1)\alpha^1_k$ and $e^{-\frac{1}{2}N_k}\log\alpha_k=o(1)$. Then we have
\begin{align*}
  \int_{B_{\alpha_ks_k}\setminus B_{s_k}}e^{u^i_k}dx\leq Ce^{-\frac{1}{2}N_k}\log\alpha_k=o(1).
\end{align*} The proof of statement $(2)$ is similar.

\

\noindent
\textbf{Step 2:} We prove  statements  $(3)$ and $(4)$.

\

We first claim there exists $\gamma_k\in [\beta_k^{\frac{1}{4}}, \beta_k^{\frac{1}{2}}]$ such that $u^1_k$ has fast decay on $\partial B_{\gamma_k s_k}$. If not, then
$$u^1_k(x)+2\log |x|\geq -C,\ \ \forall\ \beta_k^{\frac{1}{4}}s_k\leq |x|\leq \beta_k^{\frac{1}{2}} s_k,$$ for some constant $C>0$. This implies
\begin{align*}
  \int_{B_{\beta_k^{\frac{1}{2}} s_k}\setminus B_{\beta_k^{\frac{1}{4}}s_k}}e^{u^1_k}dx\geq C\log\beta_k\to +\infty,
\end{align*} which is a contradiction.
By  $(1)$, there exists $\gamma_{k,1}\to +\infty$ such that $\frac{\gamma_{k,1}}{\gamma_k}\to +\infty$, $\gamma_{k,1}=o(\beta_k^{\frac{2}{3}})$ and $u^1_k$ has fast decay in $B_{\gamma_{k,1} s_k}\setminus B_{\gamma_k s_k}$. Similarly, there exist $\gamma_{k,2}\in [\gamma_{k,1}^{\frac{1}{4}}, \gamma_{k,1}^{\frac{1}{2}}]$ and $\gamma_{k,3}\to +\infty$, such that $\frac{\gamma_{k,3}}{\gamma_{k,2}}\to +\infty$, $\gamma_{k,3}=o(\beta_k^{\frac{2}{3}})$ and $u^2_k$ has fast decay in $B_{\gamma_{k,3} s_k}\setminus B_{\gamma_{k,2} s_k}$. From the construction, it is obviously that $u^1_k$ also has fast decay in $B_{\gamma_{k,3} s_k}\setminus B_{\gamma_{k,2} s_k}$.

For $u^3_k$, a similar discussion yields there exists $\alpha_k\in [\gamma_{k,3} ^{\frac{1}{4}}, \gamma_{k,3} ^{\frac{1}{2}}]$ such that $u^3_k$ has fast decay on $\partial B_{\alpha_k s_k}$, where $\alpha_k\to+\infty$ and $\alpha_k=o(1)\beta_k$. We have proved   $(3)$. The proof of  $(4)$ is similar.

\end{proof}

\

At the end of this section, we prove a Pohozaev identity for the Toda system \eqref{equ:03}-\eqref{equ:06}, which plays an important role in our later proof. Denote $$\sigma (r,x_0;u):=\frac{1}{2\pi}\int_{B_{r}(x_0)}e^udx,\ \ \sigma^i_k(r,x_0):=\sigma (r,x_0;u^i_k) ,\ \ \sigma_k^i(r):=\sigma (r,0;u_k^i),\ i=1,2,3.$$

\begin{lem}\label{lem:Pohozaev}
Let $\Sigma^1_k\subset \Sigma_k$, $B_{r_k}(x_k)\subset B_1(0)$ and $$\dist(\Sigma_k^1,\partial B_{r_k}(x_k))=o(1)\, \dist(\Sigma_k\setminus \Sigma_k^1, \partial B_{r_k}(x_k) ).$$ Suppose $u_k^1,\ u^2_k, \ u^3_k$ have fast decay on $\partial B_{r_k}(x_k)$. Then we have the following Pohozaev identity
\begin{align}\label{equation:pohozaev}
&\big(\sigma_k^1(r_k,x_k)-\sigma_k^3(r_k,x_k)\big)^2+ \big(\sigma_k^2(r_k,x_k)-\sigma_k^3(r_k,x_k)\big)^2\notag \\&\quad =4\big(\sigma_k^1(r_k,x_k)+\sigma_k^2(r_k,x_k)+2\sigma_k^3(r_k,x_k)\big)+o(1). \end{align}
\end{lem}
\begin{proof}
Without loss of generality, we assume $x_k=0$. For any $s\in (0,1)$, by  equations \eqref{equ:03} and \eqref{equ:04}, we have Pohozaev's identities
\begin{align*}
  -s\int_{\partial B_s}\left(\big|\frac{\partial u_k^i}{\partial r}\big|^2-\frac{1}{2}\big|\nabla u_k^i \big|^2\right)=\int_{\partial B_s}se^{u_k^i}-2\int_{B_s}e^{u_k^i}-\int_{B_s}e^{u_k^3}x\cdot\nabla u_k^i,\ i=1,2.
\end{align*}
This implies
\begin{align*}
  &-s\int_{\partial B_s}\left(\bigg|\frac{\partial u_k^1}{\partial r}\bigg|^2-\frac{1}{2}|\nabla u_k^1|^2\right)-s\int_{\partial B_s}\left(\bigg|\frac{\partial u_k^2}{\partial r}\bigg|^2-\frac{1}{2}|\nabla u_k^2|^2\right)\\
  &\qquad =\int_{\partial B_s}s(e^{u_k^1}+e^{u_k^2})-2\int_{B_s}(e^{u_k^1}+e^{u_k^2})dx-\int_{B_s}e^{u_k^3}x\cdot (\nabla u_k^1+\nabla u_k^2)dx.
\end{align*}
By \eqref{equ:06} and integrating by parts, we get
\begin{align*}
  -\int_{B_s}e^{u_k^3}x\cdot (\nabla u_k^1+\nabla u_k^2)dx&=2\int_{B_s}e^{u_k^3}x\cdot \nabla u_k^3dx\\
  &=2\int_{\partial B_s}se^{u_k^3}-4\int_{B_s}e^{u_k^3}dx.
\end{align*} Then we arrived at
\begin{align}\label{equation:01}
 & -s\int_{\partial B_s}\left(\big|\frac{\partial u_k^1}{\partial r}\big|^2-\frac{1}{2}|\nabla u_k^1|^2\right)-s\int_{\partial B_s}\left(\big|\frac{\partial u_k^2}{\partial r}\big|^2-\frac{1}{2}|\nabla u_k^2|^2\right)\\&=
  \int_{\partial B_s}s(e^{u_k^1}+e^{u_k^2}+2e^{u_k^3})-2\int_{B_s}(e^{u_k^1}+e^{u_k^2}+2e^{u_k^3})dx.
\end{align}

Now we use the crucial condition, the fast decay of solutions.
Since $u_k^1,\ u^2_k,\ u^3_k$ have fast decay on $\partial B_{r_k}$, i.e. $$u^i_k(x)+2\log |x|\leq -N_k, |x|=r_k,$$ for some $N_k\to \infty$, by Lemma \ref{lem:05}, there exists $R_k^1\to\infty$ such that $$R^1_k=\frac{\dist(\Sigma_k\setminus \Sigma_k^1, \partial B_{r_k}(x_k) )}{\dist(\Sigma_k^1,\partial B_{r_k}(x_k))}o(1)$$ and
\begin{align}\label{equation:03}
  u^i_k(x)+2\log |x|\leq -N_k, r_k\leq |x|\leq r_kR^1_k,
\end{align}
for some $N_k\to \infty$. Now, it is not hard to see that we can choose $R_k\to\infty$, $R_k\leq R^1_k$, such that
\begin{align*}
  \int_{B_{r_kR_k}\setminus B_{r_k}}e^{u^i_k(x)}dx\leq Ce^{-N_k}\log R_k=o(1),
\end{align*} which implies that
\begin{align}\label{equation:02}
  \sigma_k^i(r_kR_k)=\sigma_k^i(r_k)+o(1),\ \ i=1,2,3.
\end{align}

Taking $s=r_k\sqrt{R_k}$ in \eqref{equation:01}, we have
 $$-2\int_{B_{r_k\sqrt{R_k}}}(e^{u^1_k}+e^{u_k^2}+2e^{u_k^3})dx=-4\pi \big(\sigma_k^1(r_k)+\sigma_k^2(r_k)+2\sigma_k^3(r_k)\big)+o(1)$$ and $$\int_{\partial B_{r_k\sqrt{R_k}}}s(e^{u_k^1}+e^{u_k^2}+2e^{u_k^3})=o(1),$$ since $u_k^1,\ u^2_k,\ u^3_k$ have fast decay on $\partial B_{r_kR_k}$. Thus, there holds
\begin{align}\label{equation:04}
 &-4\pi \big(\sigma_k^1(r_k)+\sigma_k^2(r_k)+2\sigma_k^3(r_k)\big)+o(1) \notag\\&=-\int_{\partial B_{r_k\sqrt{R_k}}}r\left(|\frac{\partial u^1_k}{\partial r}|^2-\frac{1}{2}|\nabla u_k^1|^2\right)-\int_{\partial B_{r_k\sqrt{R_k}}}r\left(|\frac{\partial u_k^2}{\partial r}|^2-\frac{1}{2}|\nabla u_k^2|^2\right).
\end{align}
Following arguments given in \cite{Jevnikar-Wei-Yang}, we have
 $$\nabla u_k^1(x)=-\frac{x}{|x|^2} \big(\sigma_k^1(r_k)-\sigma_k^3(r_k)\big)+\frac{o(1)}{|x|},\ \ x\in\partial B_{r_k\sqrt{R_k}},$$ and $$\nabla u_k^2(x)=-\frac{x}{|x|^2} \big(\sigma_k^2(r_k)-\sigma_k^3(r_k)\big)+\frac{o(1)}{|x|},\ \ x\in\partial B_{r_k\sqrt{R_k}}.$$
By \eqref{equation:04}, we conclude
\begin{align*}
&4 \big(\sigma_k^1(r_k)+\sigma_k^2(r_k)+2\sigma_k^3(r_k)\big)+o(1) =\big(\sigma_k^1(r_k)-\sigma_k^3(r_k)\big)^2+ \big(\sigma_k^2(r_k)-\sigma_k^3(r_k)\big)^2.
\end{align*}

\end{proof}

\section{Local blow-up behavior}\label{sec:local-blowup-behavior}

\

In this section, we will prove a local blow-up behaviour Theorem \ref{thm:01}. This is a key step in the proof of our main Theorem \ref{thm:main}. We need to do a more careful blow-up analysis in the bubbling domain. See the proof of  Proposition \ref{prop:01}.

\

By a  translation  we may assume $x^1_k=0\in\Sigma_k$ for any $k$. Denote $\tau_k:=\frac{1}{2}\dist(0,\Sigma_k\setminus \{0\})$.

\begin{thm}\label{thm:01}
Let $(u^1_k,\ u^2_k,\ u^3_k)$ be a solution of \eqref{equ:03}-\eqref{equ:06}. Then we have
\begin{itemize}
  \item[(i)]  either all $u^1_k,\ u^2_k,\ u^3_k$ have fast decay on $ \partial B_{\tau_k}$ and $$\big(\sigma_k^1(\tau_k), \sigma^2_k(\tau_k), \sigma^3_k(\tau_k)\big)\in\mathbf{V}+o(1),$$ or there exists one component $u^i_k$ with slow decay on $\partial B_{\tau_k}$ and $$\lim_{s\to 0}\lim_{k\to\infty}\big(\sigma_k^1(s\tau_k), \sigma_k^2(s\tau_k), \sigma_k^3(s\tau_k)\big)\in\mathbf{V}.$$
  \item[(ii)] There exists at least one component $u^i_k$ such that  $u^i_k$ has fast decay on $\partial B_{\tau_k}$. Moreover, if
   $u^i_k$ has fast decay on $\partial B_{\tau_k}$, then $\sigma^i_k(\tau_k)=4n+o(1)$ for some $n\in\mathds{N}\cup \{0\}$.
\end{itemize}
\end{thm}
\begin{proof}
Since $0\in\Sigma_k$, denoting $\epsilon^i_k:=u^i_k(0)$, $i=1,2,3$, by Lemma \ref{lem:01}, we have $$\max_{i=1,2,3}\epsilon^i_k=\max_{i=1,2,3}u^i_k(0)=\max_{i=1,2,3}\max_{B_{\lambda_k^1}}u^i_k(x)\to\infty.$$ Let $$\epsilon_k:=\min\{e^{-\frac{1}{2}\epsilon^1_k},e^{-\frac{1}{2}\epsilon^2_k},e^{-\frac{1}{2}\epsilon^3_k}\}$$ and $$v^i_k(x)=u^i_k(\epsilon_k x)+2\log\epsilon_k,\ \ |x|\leq\frac{\tau_k}{\epsilon_k},\ i=1,2,3.$$ It is clear that ${\tau_k}\slash {\epsilon_k}\geq {\lambda^1_k}\slash {\epsilon_k}\to+\infty$. It is also  easy to see that $v^i_k(x)\leq 0$ and
\begin{align}
\begin{cases}
-\Delta v_k^1=e^{v_k^1}-e^{v_k^3}, \\
  -\Delta v_k^2=e^{v_k^2}-e^{v_k^3}, \\
  -\Delta v_k^3=-\frac{1}{2}e^{v_k^1}-\frac{1}{2}e^{v_k^2}+ e^{v_k^3},
\end{cases}\ \hbox{ in }\ B_{\lambda^1_k/\epsilon_k}(0).\end{align}

In view of  $$\epsilon^1_k+\epsilon^2_k+2\epsilon^3_k =u^1_k(0)+u^2_k(0)+2u^3_k(0)=0,$$ we know that  exactly one of the following possibilities holds:
\begin{itemize}
  \item[(1)] $(\epsilon^1_k,\epsilon^2_k,\epsilon^3_k)\to (+\infty,+\infty,-\infty)$, $\sup_k|\epsilon^1_k-\epsilon_k^2|\leq C$;
  \item[(2)]  $(\epsilon^1_k,\epsilon^2_k,\epsilon^3_k)\to (+\infty,+\infty,-\infty)$, $|\epsilon^1_k-\epsilon_k^2|\to +\infty $;
  \item[(3)] $(\epsilon^1_k,\epsilon^2_k,\epsilon^3_k)\to (+\infty,-\infty,+\infty)$, $\sup_k|\epsilon^1_k-\epsilon_k^3|\leq C$;
  \item[(4)]$(\epsilon^1_k,\epsilon^2_k,\epsilon^3_k)\to (+\infty,-\infty,+\infty)$, $|\epsilon^1_k-\epsilon_k^3|\to +\infty $;
  \item[(5)] There exists one component $\epsilon_k^i\to +\infty$ while the other ones have uniformly upper bound, i.e. $\epsilon_k^j\leq C$, for $j\neq i$.
\end{itemize}

\

Now, for the above cases, we claim that there exists $R^1_k\to +\infty$ such that $R^1_k=o(1)\frac{\tau_k}{\epsilon_k}$ and $u^1_k,\ u^2_k,\ u^3_k$ have fast decay on $\partial B_{\epsilon_k R^1_k}$ and $$(\sigma^1_k(\epsilon_k R^1_k), \sigma^2_k(\epsilon_k R^1_k), \sigma^3_k(\epsilon_k R^1_k))=(4n_1,4n_2,4n_3)+o(1)\in\mathbf{V}+o(1).$$
We only give the proof for the third case and the other cases are similar or easier. For case $(3)$, without loss of generality, we assume $\epsilon_k=e^{-\frac{1}{2}\epsilon^1_k}$, i.e. $$\epsilon_k^1\geq \epsilon^3_k,\ \ \epsilon_k^1-\epsilon^3_k\leq C.$$ It is easy to see that $$v^i_k\leq 0,\ v^1_k(0)=0,\ v^3_k(0)\geq -C,\ v^2_k(0)\to -\infty.$$ By the standard elliptic theory, we have that $v^2_k\to -\infty$ over all compact subsets of  $\R^2$ and $$(v^1_k,v^3_k)\to (v^1,v^3), \ \ \hbox { in }\ \ C^2_{loc}(\R^2),$$ where the limit $(v^1,v^3)$ satisfies $v^1(0)=0$, $-C\leq v^3(0)\leq 0$ and the following Liouville  system
\begin{align}
\begin{cases}
-\Delta v^1=e^{v^1}-e^{v^3}, \\
  -\Delta v^3=-\frac{1}{2}e^{v^1}+e^{v^3},
\end{cases}\ \hbox { in }\ \R^2,\end{align} with $$\int_{\R^2}e^{v^1}dx+\int_{\R^2}e^{v^3}dx\leq C<\infty.$$ By the classification result Theorem \ref{thm:classification-1},  there holds $$\int_{\R^2}e^{v^1}dx =32\pi,\ \ \int_{\R^2}e^{v^2}dx =24\pi\ \ and \ \ v^i(x)=-4\log |x|+O(1),\ |x|>4,\ i=1,2.$$ Moreover, we can choose a sequence of $R^1_k\to+\infty$ such that $R^1_k=o(1)\frac{\tau_k}{\epsilon_k}$ and  $$\int_{B_{R^1_k}}e^{v^1_k}dx=32\pi+o(1),\ \ \int_{B_{R^1_k}}e^{v^2_k}dx=24\pi+o(1)$$ and $v^1_k,\ v^2_k,\ v^3_k$ have fast decay on $\partial B_{R^1_k}$. It is easy to see that $\sigma^i_k(\epsilon_k R^1_k)=4n_i+o(1)$ for some $n_i\in\mathds{N}\cup \{0\}$, $i=1,2,3$. Since $u^1_k,\ u^2_k,\ u^3_k$ have fast decay on $\partial B_{\epsilon_kR^1_k}$, by Lemma \ref{lem:Pohozaev}, we conclude that the claim follows.

With the help of the above claim, the conclusion of the theorem follows immediately from Proposition \ref{prop:01} below.

\end{proof}

Before stating Proposition \ref{prop:01}, we give two lemmas which will be used in the proof of Proposition \ref{prop:01}.

Define $$\overline{u}^i_k(r):=\frac{1}{2\pi r}\int_{\partial B_r(0)} u^i_k,\ \qquad  i=1,2,3.  $$
By equations \eqref{equ:03}-\eqref{equ:05}, we get
\begin{align}
\frac{d}{dr} \overline{u}^1_k(r)&=\frac{-\sigma_k^1(r)+\sigma_k^3(r)}{r},\ \ \quad \qquad \quad   \frac{d}{dr} \overline{u}^2_k(r)=\frac{-\sigma_k^2(r)+\sigma_k^3(r)}{r},\label{equation:12}\\ \frac{d}{dr} \overline{u}^3_k(r)&=\frac{\frac{1}{2}\sigma_k^1(r)+ \frac{1}{2}\sigma_k^2(r) -\sigma_k^3(r)}{r}.\label{equation:13}\end{align}

\begin{lem}\label{lem:04}
Suppose that  $\big(\sigma^1_k(r),\sigma^2_k(r),\sigma^3_k(r)\big)=(4n_1,4n_2,4n_3)+o(1)\in \mathbf{V}+o(1)$ holds for $r$. Then there exists $i\in\{1,2,3\}$ and $
j\neq i,\ j\in\{1,2,3\}$ such that
\begin{equation} \label{eq_add_c}
\frac{d}{dr} \overline{u}^i_k(r)\leq \frac{-4+o(1)}{r}
\quad \hbox{ and } \quad   \frac{d}{dr} \overline{u}^j_k(r)\geq \frac{1+o(1)}{r}.\end{equation}
\end{lem}

\begin{proof}

First we have by assumption $$\frac{d}{dr} \overline{u}^1_k(r)=\frac{4(n_3-n_1)+o(1)}{r},\ \  \frac{d}{dr} \overline{u}^2_k(r)=\frac{4(n_3-n_2)+o(1)}{r} $$
and
 $$\frac{d}{dr} \overline{u}^3_k(r)=4\frac{\frac{1}{2}n_1+ \frac{1}{2}n_2 -n_3+o(1)}{r}.$$
 Suppose that the first conclusion does not hold. Then
 it is easy to see that
$$n_3\geq n_1, \quad  n_3\geq n_2, \quad  n_1+n_2-2n_3 \ge -1.$$
It is clear that  there are only 3 possibilities: $n_3=n_1=n_2$,  $n_3=n_1=n_2+1$ or $n_3=n_1+1=n_2$. All contradict
 the fact $(n_1, n_2, n_3)\in \mathbf{V}.$
 This proves the first inequality.

The second inequality  follows easily from the first one and
\begin{equation*}
\frac{d}{dr} \overline{u}^1_k(r)+\frac{d}{dr} \overline{u}^2_k(r)+2\frac{d}{dr} \overline{u}^3_k(r)=0.\end{equation*}
\end{proof}

\begin{lem}\label{lem:03}
  Suppose there exist $R^1_k$ and $R^2_k$ with $R^1_k\leq R^2_k\leq \tau_k/\epsilon_k$
  satisfying  the following properties that
  $\big(\sigma^1_k(\epsilon_kR^1_k), \sigma^2_k(\epsilon_kR^1_k), \sigma^3_k(\epsilon_kR^1_k)\big)\in \mathbf{V}+o(1)$ and
    $u^i_k$ ($i=1,2,3)$ have fast decay in $B_{R^2_k}\backslash B_{R^1_k}$, i.e.,
   $$u^i_k(x)\leq -2\log |x|-N_k,\ \ \forall\ R^1_k\leq |x|\leq R^2_k,\ i=1,2,3,$$ for some $N_k\to+ \infty$. Then $$\sigma^i_k(\epsilon_kR^2_k)=\sigma^i_k(\epsilon_kR^1_k)+o(1),\  \quad i=1,2,3.$$
\end{lem}

\begin{proof}

Suppose the conclusion is false. Then there exists $\delta_0>0$ such that $$\max_{i=1,2,3}\sigma^i_k(\epsilon_kR^2_k)-\sigma^i_k(\epsilon_kR^1_k)>\delta_0,$$ when $k$ is big enough.
For any $\delta_0 >\delta \in (0,\frac{1}{100})$ which is small enough, we can choose $L_k\in [R^1_k,R^2_k]$ such that \begin{equation}
\label{eq:add_1}
\max_{i=1,2,3}\sigma^i_k(\epsilon_kL_k)-\sigma^i_k(\epsilon_kR^1_k)=\delta.\end{equation}
By Lemma \ref{lem:04},  there exist $i, j\in\{1,2,3\}$ with  $i\neq j$ such that
\begin{equation}\label{3.4_a} \frac{d}{dr} \overline{u}^i_k(r)\leq \frac{-4+o(1)}{r}\ \quad \hbox{ and }\quad   \frac{d}{dr} \overline{u}^j_k(r)\geq \frac{1+o(1)}{r},\ \
\quad\hbox {for }r=\epsilon_kR^1_k.\end{equation}
Now, we claim:
\begin{equation}
\label{eq_add_d}
\sigma^i_k(\epsilon_kL_k)=\sigma^i_k(\epsilon_kR^1_k)+o(1),\ \sigma^j_k(\epsilon_kL_k)=\sigma^j_k(\epsilon_kR^1_k)+o(1). \end{equation}

We prove the claim.
 From \eqref{equation:12}-\eqref{equation:13}, \eqref{eq:add_1} and \eqref{3.4_a}, we have
 \begin{align}\label{eq:add_2}
   r\frac{d}{dr} \overline{u}^i_k(r)&\leq \epsilon_kR^1_k\frac{d}{dr}\overline{u}^i_k(\epsilon_kR^1_k)+3\delta\le-4+3\delta+o(1)\leq -(2+a),\quad \ \forall\ \epsilon_kR^1_k\leq r\leq \epsilon_k L_k,
 \end{align} for some $a>0$ small. It follows  by integrating that
 \begin{align*}
  \overline{u}^i_k(r)\leq \overline{u}^i_k(\epsilon_kR^1_k)-(2+a)\log\frac{r}{\epsilon_kR^1_k},\ \quad \forall\ \epsilon_kR^1_k\leq r\leq \epsilon_k L_k.
 \end{align*} This, together with that fact that $u^i_k$ has fast decay on $\partial B_{\epsilon_kR^1_k}$,
 yields
 \begin{align*}
   \sigma_k^i(\epsilon_kL_k)-\sigma^i_k(\epsilon_kR^1_k)& =\int_{B_{\epsilon_kL_k}\setminus B_{\epsilon_kR^1_k}}e^{u^i_k(x)}dx
   \leq C \int_{B_{\epsilon_kL_k}\setminus B_{\epsilon_kR^1_k}}e^{\overline{u}^i_k(x)}dx\\
   &\leq C e^{\overline{u}^i_k(\epsilon_kR^1_k)}(\epsilon_kR^1_k)^2\leq Ce^{-N_k},
 \end{align*}  where the first and third inequalities follow from Lemma \ref{lem:02}.
 Still from \eqref{equation:12}-\eqref{equation:13}, \eqref{eq:add_1} and \eqref{3.4_a},
 we have
 \begin{align}\label{equation:09}
 r  \frac{d}{dr} \overline{u}^j_k(r)\geq  \epsilon_kR^1_k \frac{d}{dr} \overline{u}^j_k(\epsilon_kR^1_k)-3\delta\geq 1-3\delta+o(1)\geq b,\  \quad  \forall\ \epsilon_kR^1_k\leq |x|\leq \epsilon_k L_k,
 \end{align} for some $b>0$,  which implies
 \begin{align*}
  \overline{u}^j_k(r)\leq \overline{u}^j_k(\epsilon_k L_k)-b\log\frac{\epsilon_kL_k}{r},\ \ \forall\ \epsilon_kR^1_k\leq r\leq \epsilon_k L_k.
 \end{align*}
 Since $u^j_k$ has fast decay on $\partial B_{\epsilon_kL_k}$, it follows
 \begin{align}\label{inequ:01}
   \sigma^j_k(\epsilon_kL_k)-\sigma^j_k(\epsilon_kR^1_k)& =\int_{B_{\epsilon_kL_k}\setminus B_{\epsilon_kR^1_k}}e^{u^j_k(x)}dx\notag
   \leq C \int_{B_{\epsilon_kL_k}\setminus B_{\epsilon_kR^1_k}}e^{\overline{u}^j_k(x)}dx\notag \\
   &\leq C e^{\overline{u}^j_k(\epsilon_kL_k)}(\epsilon_kL_k)^2\leq Ce^{-N_k},
 \end{align}  where the first and third inequality follow from Lemma \ref{lem:02}. Thus the claim holds.

 Since all  $u^i_k$, $i=1,2,3$ have  fast decay on $\partial B_{\epsilon_kL_k}$ and also on $\partial B_{\epsilon R_k}$
  by Lemma \ref{lem:Pohozaev} we have  the local Pohozaev identity, i.e.,
  $$\big(\sigma_k^1(\epsilon_kL_k)-\sigma^3_k(\epsilon_kL_k)\big)^2+ \big(\sigma^2_k(\epsilon_kL_k)-\sigma^3_k(\epsilon_kL_k)\big)^2=4\big(\sigma^1_k(\epsilon_kL_k)+\sigma^2_k(\epsilon_kL_k) +2\sigma^3_k(\epsilon_kL_k)\big)+o(1),$$
  and
   $$\big(\sigma_k^1(\epsilon_kR_k)-\sigma^3_k(\epsilon_kR_k)\big)^2+ \big(\sigma^2_k(\epsilon_kR_k)-\sigma^3_k(\epsilon_kR_k)\big)^2=4
   \big(\sigma^1_k(\epsilon_kR_k)+\sigma^2_k(\epsilon_kR_k) +2\sigma^3_k(\epsilon_kL_k)\big)+o(1).$$
   Let $l \in\{ 1,2,3\} \backslash \{i,j\}$. By \eqref{eq_add_d} we know that
   $\sigma^l(\epsilon_kL_k)$ and
    $\sigma^l(\epsilon_kR_k)$ satisfy a quadratic equation with coefficients
    which differ only  by $o(1)$. It is easy to see that such a  quadratic equation has roots very close integers. Hence
    the difference
     between $\sigma^l(\epsilon_kL_k)$ and
    $\sigma^l(\epsilon_kR_k)$  is either $o(1)$ or bigger than $\frac 12$. By \eqref{eq:add_1} we conclude
     $\sigma^l _k(\epsilon_kL_k)- \sigma^l_k(\epsilon_kR_k)=
    +o(1)$,
  a contradiction to \eqref{eq:add_1}.
 Hence we finish the proof.
\end{proof}

\begin{rem}\label{lem3.4} Lamma \ref{lem:03} holds true, if   $\big(\sigma^1_k(\epsilon_kR^1_k), \sigma^2_k(\epsilon_kR^1_k), \sigma^3_k(\epsilon_kR^1_k)\big)\in \mathbf{V}+o(1)$
	is replaced by $\dist ( \big(\sigma^1_k(\epsilon_kR^1_k), \sigma^2_k(\epsilon_kR^1_k), \sigma^3_k(\epsilon_kR^1_k)\big),  \mathbb{V} ) < \delta$
	for a  small enough constant $\delta>0$. In this case, the  inequalities \eqref{eq_add_c} in   Lemma \label{lem.04} is changed to
 $$\frac{d}{dr} \overline{u}^i_k(r)\leq \frac{-4+4\delta +o(1)}{r}
	\quad \hbox{ and } \quad   \frac{d}{dr} \overline{u}^j_k(r)\geq \frac{1-\delta+o(1)}{r},$$
	while the arguments for \eqref{eq:add_2} and \eqref{equation:09} still work. Hence \eqref{eq_add_d} holds. Then
	 the local Pohozaev identity implies the conclusion.
	\end{rem}

\

Now, we state Proposition \ref{prop:01}.

\begin{prop}\label{prop:01}
  Suppose there exists a sequence  $R^1_k\to+\infty$ such that
  $R^1_k=o(1) \tau_k/\epsilon_k$ and $$\big(\sigma^1_k(\epsilon_kR^1_k), \sigma^2_k(\epsilon_kR^1_k), \sigma^3_k(\epsilon_kR^1_k)\big)=(4n_1,4n_2,4n_3)+o(1)\in\mathbf{V}+o(1).$$ Then we have:
\begin{itemize}
  \item[(1)] For any $R^2_k\geq R^1_k$ with ${R^2_k}\slash{R^1_k}\to +\infty$ and  $R^2_k=o(1) \tau_k/\epsilon_k$, if $u^i_k (i=1,2,3)$ have fast decay on $ \partial B_{\epsilon_kR^2_k}$, then there holds $$\big(\sigma^1_k(\epsilon_kR^2_k), \sigma^2_k(\epsilon_kR^2_k), \sigma^3_k(\epsilon_kR^2_k)\big)\in\mathbf{V}+o(1).$$
  \item[(2)] On $\partial B_{\tau_k}$, either $u^1_k,\ u^2_k,\ u^3_k$ have fast decay on $ B_{\tau_k}$ and $$\big(\sigma^1_k(\tau_k), \sigma^2_k(\tau_k), \sigma^3_k(\tau_k)\big)\in\mathbf{V}+o(1),$$ or there exists
  one component $u^i_k$ with slow decay on $\partial B_{\tau_k}$ and $$\lim_{s\to 0}\lim_{k\to\infty}\big(\sigma^1_k(s\tau_k), \sigma^2_k(s\tau_k), \sigma^3_k(s\tau_k)\big)\in\mathbf{V}.$$
  \item[(3)] There exists at least one component $u^i_k$ such that  $u^i_k$ has fast decay on $\partial B_{\tau_k}$. 
Moreover,  if $u^i_k$  has fast decay on $\partial B_{\tau_k}$, then $\sigma^i_k(\tau_k)=4n+o(1)$ for some $n\in \mathds{N}$.
\end{itemize}

\end{prop}

\begin{proof} (1)
Suppose there exists $\delta_0>0$ such that \begin{equation}
\label{contro}
|\big(\sigma^1_k(\epsilon_kR^2_k), \sigma^2_k(\epsilon_kR^2_k), \sigma^3_k(\epsilon_kR^2_k)\big)-\big(\sigma^1_k(\epsilon_kR^1_k), \sigma^2_k(\epsilon_kR^1_k), \sigma^3_k(\epsilon_kR^1_k)\big)|>
\delta_0,\end{equation} when $k$ is big, otherwise we have done.
By Lemma \ref{lem:04}, we consider only the case that
\begin{equation}\label{equation:23}
\frac{d}{dr} \overline{u}^1_k(r)=\frac{-4n_1+4n_3}{r}\leq \frac{-4+o(1)}{r},\ \ \frac{d}{dr} \overline{u}^3_k(r)=\frac{2n_1+2n_2-4n_3}{r}\geq \frac{1+o(1)}{r},\ \ r=\epsilon_kR_k^1.
\end{equation} The other cases can be studied similarly.
 We obtain a contradiction in several steps.

\

\noindent\textbf{Step  1.}  By fixing $\delta \in (0,\min\{\delta_0,\frac{1}{100}\})$  we can  find
a special radius $L_k\in [R^1_k,R^2_k]$ with $L_k=o(1) R_k^2$ such that  at least for one $i$,  $u^i_k$ has slow decay on $\partial B_{\epsilon_kL_k}$ and $$\frac{\delta}{2}\leq \max_{i=1,2,3}\sigma^i_k(\epsilon_kL_k)-\sigma^i_k(\epsilon_kR^1_k)\leq\delta\leq\frac{1}{100}.$$
Moreover, $u^1_k$ has fast decay on $B_{\epsilon_kL_k}\backslash B_{\epsilon_kR_k}$.

\

From  \eqref{contro} it is clear  that there exist  $L^1_k,\ L^2_k\in [R^1_k,R^2_k]$ such that
\begin{equation}\label{equation:15}
\max_{i=1,2,3}\sigma^i_k(\epsilon_kL^1_k)-\sigma^i_k(\epsilon_kR^1_k)=\frac{\delta}{2},\ \ \max_{i=1,2,3}\sigma^i_k(\epsilon_kL^2_k)-\sigma^i_k(\epsilon_kR^1_k)=\delta.
\end{equation}
We can show that there exist  $i$ and $L_k\in [L^1_k,L^2_k]$ such that $u^i_k$ has slow decay on $\partial B_{\epsilon_kL_k}$. Otherwise, $u^1_k,\ u^2_k,\ u^3_k$ have fast decay in $B_{\epsilon_kL^2_k}\setminus B_{\epsilon_kL^1_k}$. From the assumption of the Proposition and \eqref{equation:15} we have that $\dist \big( (\s_k^1(\epsilon_kL_k^1),\s_k^2(\epsilon_kL_k^1),\s_k^2(\epsilon_kL_k^1)
),  \mathbb{V}\big ) <2\delta$. Hence
 we can use Remark \ref{lem3.4}
to show that
 $$\max_{i=1,2,3}\sigma^i_k(\epsilon_kL^2_k)-\sigma^i_k(\epsilon_kL^1_k)=o(1),$$  which contradicts  \eqref{equation:15}. Hence
  $u^i_k$ has slow decay on $\partial B_{\epsilon_kL_k}$. By Lemma \ref{lem:05}  we know that $L_k=o(1) R^2_k$.
From \eqref{equation:15}, we have
\begin{equation}\label{equation:06}
\frac{\delta}{2}\leq \max_{i=1,2,3}\sigma^i_k(\epsilon_kL_k)-\sigma^i_k(\epsilon_kR^1_k)\leq\delta\leq\frac{1}{100}.\end{equation}
By \eqref{equation:12}, \eqref{equation:06} and \eqref{equation:23}, we have $$\frac{d}{dr}\overline{u}^1_k\leq \frac{-4+3\delta+o(1)}{r},\ \quad\epsilon_kR^1_k\leq r\leq \epsilon_kL_k,$$
which implies $u^1_k$ has fast decay in $B_{\epsilon_kL_k}\setminus B_{\epsilon_kR^1_k}$.

\

\noindent
\textbf{Step 2.} We claim that there exist $\delta_1>0$ and a  constant $N$ such that
\begin{equation}\label{equation:08}
  \sigma^2_k(N\epsilon_kL_k)-\sigma^2_k(\epsilon_kL_k)+\sigma^3_k(N\epsilon_kL_k)-\sigma^3_k(\epsilon_kL_k)\geq \delta_1,
\end{equation} when $k$ is big enough.

\

If not, then there exists  $\tilde{ R}_k\to+\infty$ with $\tilde{ R}_k\epsilon_kL_k\leq \epsilon_kR^2_k$, such that
\begin{equation}\label{equation:11}
\sigma^2_k(\tilde{ R}_k\epsilon_kL_k)-\sigma^2_k(\epsilon_kL_k)+\sigma^3_k(\tilde{ R}_k\epsilon_kL_k)-\sigma^3_k(\epsilon_kL_k)=o(1).\end{equation}
Since by Step 1  $u^1_k$ has fast decay in $B_{\epsilon_kL_k}\setminus B_{\epsilon_kR^1_k}$,
as in the proof of Lemma \ref{lem:03}
 we have $$\sigma^1_k(\epsilon_kL_k)-\sigma^1_k(\epsilon_kR^1_k)=o(1).$$ By Lemma \ref{lem:05}, there exists $R_k\to +\infty$
 with $R_k=o(1)\tilde{ R}_k$ such that $u^1_k$ has fast decay in $B_{R_k\epsilon_kL_k}\backslash B_{\epsilon_kL_k}$ and $$\sigma^1_k(R_k\epsilon_kL_k)-\sigma^1_k(\epsilon_kL_k)=o(1).$$
 Hence  we have that $u^1_k$ has fast decay in $B_{R_k\epsilon_kL_k}\backslash B_{\epsilon_kR^1_k}$ and
\begin{equation}\label{equation:10}
\sigma^1_k(R_k\epsilon_kL_k)-\sigma^1_k(\epsilon_kR^1_k)=o(1).\end{equation}

We claim first that  $u^3_k$ also  has fast decay on $\partial B_{\epsilon_kL_k}$. If not, by  \eqref{equation:09}, \eqref{equation:06}, \eqref{equation:10}, \eqref{equation:11} and \eqref{equation:23}, we have
 \begin{align}\label{equation:21}
   \frac{d}{dr} \overline{u}^3_k(r)\geq \frac{1-3\delta+o(1)}{r}\geq \frac{b}{r},\  \ \forall\ \epsilon_kR^1_k\leq |x|\leq R_k\epsilon_k L_k,
 \end{align} for some $b>0$, which implies $$\overline{u}^3_k(r)\geq \overline{u}^3_k(\epsilon_kL_k)+b\log\frac{r}{\epsilon_kL_k},    \ \forall\ \epsilon_kR^1_k\leq |x|\leq R_k\epsilon_k L_k$$
 and $$\int_{B_{R_k\epsilon_kL_k}\setminus B_{\epsilon_kL_k}}e^{u^3_k}dx\geq C\int_{B_{R_k\epsilon_kL_k}\setminus B_{\epsilon_kL_k}}e^{\overline{u}^3_k}dx\geq CR_k^{2+b}\to+\infty,$$ as $k\to\infty$, since $u^3_k$ has slow decay on $\partial B_{\epsilon_kL_k}$, where  we have used  Lemma \ref{lem:02} in the first inequality.
 This is also a contradiction.

  Since  now $u^3_k$ has fast decay on $\partial B_{\epsilon_kL_k}$ and satisfies \eqref{equation:21}, by the same computation given in  \eqref{inequ:01}, we get
 \begin{equation}
   \sigma^3_k(\epsilon_kL_k)-\sigma^3_k(\epsilon_kR^1_k)=o(1),
 \end{equation} which implies (from \eqref{equation:11})
 \begin{equation}\label{equation:16}
   \sigma^3_k(R_k\epsilon_kL_k)-\sigma^3_k(\epsilon_kR^1_k)=o(1).
 \end{equation}

By Lemma \ref{lem:05}, there exists $l_k\to\infty$ and $l_k=o(1)R_k$ such that  $u^1_k,\ u^2_k,\ u^3_k$ have fast decay on $\partial B_{l_k\epsilon_kL_k}$. Hence the local Pohozaev identity holds at $r=l_kr_k\epsilon _k$.
 Noting that
\begin{align*}
\sigma^1_k(R_k\epsilon_kL_k)-\sigma^1_k(\epsilon_kR^1_k)&=o(1),\\
\sigma^2_k(R_k\epsilon_kL_k)-\sigma^2_k(\epsilon_kR^1_k)&\leq\delta +o(1),\\
\sigma^3_k(R_k\epsilon_kL_k)-\sigma^3_k(\epsilon_kR^1_k)&=o(1)
\end{align*}  and  $\delta <\frac{1}{100}$,
 by  the Pohozaev identity \eqref{equation:pohozaev} $$\big(\sigma^1_k(r)-\sigma^3_k(r)\big)^2+ \big(\sigma^2_k(r)-\sigma^3_k(r)\big)^2=4\big(\sigma^1_k(r)+\sigma^2_k(r)+
2\sigma^3_k(r)\big)+o(1),\ \ r=l_k\epsilon_kL_k,$$  we have $$\sigma^i_k(l_k\epsilon_kL_k)-\sigma^i_k(\epsilon_kR^1_k)=o(1),\ i=1,2,3.$$
This is a contradiction to \eqref{equation:06}. Thus, \eqref{equation:08} holds.

\

\noindent
\textbf{Step 3:} There exists $\alpha_k\to +\infty$ such that $\alpha_kL_k=o(1)R^2_k$, $u^1_k,\ u^2_k,\ u^3_k$ have fast decay on $\partial B_{\alpha_k\epsilon_kL_k}$ and $$\big(\sigma^1_k(\alpha_k\epsilon_kL_k),\sigma^2_k(\alpha_k\epsilon_kL_k),\sigma^3_k(\alpha_k\epsilon_kL_k)\big)=(4n_1,4n_2'',4n_3'')+o(1)\in \mathbf{V}+o(1),$$ where $n_2''\geq n_2,\ n_3''\geq n_3,\ n_2''+n_3''>n_2+n_3$.

\
This is the main step in the proof.
Since $u^1_k$ has fast decay on $\partial B_{\epsilon_kL_k}$,
by Lemma \ref{lem:05}, there exists $R_k\nearrow+\infty$ (for the  simplicity of notation still denoted by $R_k$) such that $R_kL_k=o(1)R^2_k$ and $u^1_k$ still has fast decay in $B_{R_k\epsilon_kL_k}\setminus B_{\epsilon_kL_k}$ and $$\sigma_1(R_k\epsilon_kL_k)-\sigma_1(\epsilon_kL_k)=o(1).$$

Now, we do the blow-up argument in $B_{R_k\epsilon_k L_k}$. Let $$w^i_k(x):=u^i_k(\epsilon_kL_kx)+2\log(\epsilon_kL_k),\ \quad |x|\leq R_k,\ \  i=1,2,3.$$ By Lemma \ref{lem:01}, we know $$\max_{i=1,2,3}u^i_k(x)+2\log |x|\leq C,\ \quad  \forall \ |x|\leq \tau_k.$$ Then it is easy to see that  for any $ a>0$ there is a constant $C(a)>0$ such that  $$\max_{i=1,2,3}w^i_k(x)\leq C(a),\ \ 0<a\leq |x|\leq R_k;\ \ w^1_k(x)\to -\infty, \ \ \forall \ a \leq |x|\leq R_k.$$
Noticing that $$\frac 1{2\pi}\sum_{i=2}^3\int_{B_N\setminus B_1}e^{w^i_k}dx\geq \delta_1 $$ from \eqref{equation:08} and that  $w^i_k,\ i=1,2,3$ also satisfy equations  \eqref{equ:03}-\eqref{equ:05},
by the standard elliptic theory we have  that $w^1_k(x)\to -\infty,$ locally uniformly in $\R^2\setminus \{0\}$ and one of the following statements holds:
\begin{itemize}
  \item[(a)] $w^2_k\to -\infty$  locally uniformly in $\R^2\setminus \{0\}$ and $w^3_k\to w^3$ in $C^2_{loc}(\R^2\setminus \{0\})$  where $w^3$ satisfies
  \begin{equation}
  \label{eq_addf1}
 -\Delta w^3=-4\pi b_1\delta_0+e^{w^3}\ \quad \hbox{ in }\ \R^2, \end{equation}
  \item[(b)] $w^3_k\to -\infty$  locally uniformly in $\R^2\setminus \{0\}$ and $w^2_k\to w^2$ in $C^2_{loc}(\R^2\setminus \{0\})$  where $w^2$ satisfies $$-\Delta w^2=-4\pi b_2\delta_0+e^{w^2}\ \ in\ \R^2,$$
  \item[(c)] $(w^2_k,w^3_k)\to (w^2, w^3)$ in $C^2_{loc}(\R^2\setminus \{0\})$  where $(w^2,w^3)$ satisfies
  \begin{align}\label{eq_add_f2}
    \begin{array}{rcl}
    -\Delta w^2&=& -4\pi b_2\delta_0+e^{w^2}-e^{w^3},\\
      -\Delta w^3&=& -4\pi b_1\delta_0-\frac{1}{2}e^{w^2}+e^{w^3},
    \end{array} \quad \hbox{ in } \ \R^2,
  \end{align}
\end{itemize}
where $b_1=n_1+n_2-2n_3$, $b_2 =2(n_3-n_2)$ and $\delta_0$ is the Dirac measure at the point $0$. One can check that in  each case $b_1$ or/and $b_2$ are non-negative. In fact,
in case (a) that $b_1$ is nonnegative follows from \eqref{equation:21}.
In case (b)
 $b_2$ is also  nonnegative,  otherwise,  $$2(n_3-n_2)\leq -2.$$ By \eqref{equation:12}, we get $$\frac{d}{dr} \overline{u}^2_k(r)=\frac{-\sigma^2_k(r)+\sigma^3_k(r)}{r}=\frac{-4n_2+4n_3+o(1)}{r}\leq\frac{-4+o(1)}{r},\  \ r=\epsilon_kR^1_k.$$ Then we can prove  that $u^2_k$ has the similar properties as $u^1_k$. This implies that $w^2_k\to -\infty$  locally uniformly in $\R^2\setminus \{0\}$ which is a contradiction. For case (c), it is
 similar.

We now  claim: for any $s_k\in (\epsilon_kR^1_k,\epsilon_kL_k)$ with  $s_k=o(1)\epsilon_kL_k$,  there holds
\begin{equation}\label{equation:20}
  \sigma^i_k(s_k)-\sigma^i_k(\epsilon_kR^1_k)=o(1),\ \ i=1,2,3.
\end{equation} To prove this claim, we first assume $u^1_k, u^2_k, u^3_k$ have fast decay on $\partial B_{s_k}$.
Since $u^3_k$ has fast decay on $\partial B_{s_k}$ and \eqref{equation:09} holds for $\epsilon_kR^1_k\leq |x|\leq s_k$, by \eqref{inequ:01}, we have $$\sigma^3_k(s_k)-\sigma^3_k(\epsilon_kR^1_k)=o(1).$$  In view of  $$\sigma^1_k(s_k)-\sigma^1_k(\epsilon_kR^1_k)=o(1),\  \ \sigma^2_k(s_k)- \sigma^2_k(\epsilon_kR^1_k)\leq \delta,$$ where $\delta$ is chosen as before, by the Pohozaev identity $$\big(\sigma^1_k(r)-\sigma^3_k(r)\big)^2+ \big(\sigma^2_k(r)-\sigma^3_k(r)\big)^2=4\big(\sigma^1_k(r)+\sigma^2_k(r)+2\sigma^3_k(r)\big)+o(1),\ \ r=s_k,$$ we get that \eqref{equation:20} holds. If there exists $i\in\{1,2,3\}$ such that $u^i_k$ has slow decay on $\partial B_{s_k}$, since $s_k=o(1)\epsilon_kL_k$, by Lemma \ref{lem:05}, there exists $s_k'$ such that ${s_k'}\slash{s_k}\to +\infty$ and $s_k'=o(1)\epsilon_kL_k$ and $u^1_k, u^2_k, u^3_k$ have fast decay on $\partial B_{s'_k}$. Thus, the same argument implies that
 $$\sigma^i_k(s'_k)-\sigma^i_k(\epsilon_kR^1_k)=o(1),\ \ i=1,2,3,$$ and hence \eqref{equation:20} follows.

By the classification results in \cite{Lin-Yang-Zhong}, for above three cases,  we have $$\big(\int_{\R^2}e^{w^2}dx,\int_{\R^2}e^{w^3}dx\big)=(8\pi n'_2, 8\pi n'_3),\ \ n'_1, \ n'_2\in\mathds{N}\cup\{0\}.$$  Choose a sequence of $\gamma_k\nearrow +\infty$ such that,
\begin{eqnarray}
\label{eq_add_e1}
\|w^2_k-w^2\|_{C^2(B_{\gamma_k}\setminus B_{\frac{1}{\gamma_k}})}+ \|w^3_k-w^3\|_{C^2(B_{\gamma_k}\setminus B_{\frac{1}{\gamma_k}})}=o(1)
\end{eqnarray}
and
\begin{eqnarray}
\label{eq_add_e2}
\big(\int_{B_{\gamma_k}\setminus B_{\frac{1}{\gamma_k}}}e^{w_k^2}dx,\int_{B_{\gamma_k}\setminus B_{\frac{1}{\gamma_k}}}e^{w_k^3}dx\big)=(8\pi n'_2, 8\pi n'_3)+o(1).
\end{eqnarray}

	Using Lemma \ref{lem:05} again, we get that there exists $\alpha_k\nearrow+\infty$ with $\alpha_k=o(1) \min\{\gamma_k,R_k,\ \frac{L_k}{R^1_k}\}$, such that $u^1_k,u^2_k,u^3_k$ have fast decay on $\partial B_{\alpha_k\epsilon_kL_k}$. Now we  estimate $\int_{B_{\alpha_k}}e^{w_k^i}dx$ for $ i=2,3$ as follows
	$$\int_{B_{\alpha_k}}e^{w_k^2}dx= \int_{B_{\alpha_k}\setminus B_{\frac{1}{\alpha_k}}}e^{w_k^2}dx+ \int_{ B_{\frac{1}{\alpha_k}}}e^{w_k^2}dx:=I+II.$$
	Since $ \alpha_k= o(1) \gamma_k$, by \eqref{eq_add_e1} we have
	$$I =  \int_{B_{\alpha_k}\setminus B_{\frac{1}{\alpha_k}}}e^{w^2}dx +o(1)  = 8\pi n_2'+o(1).
	$$
	By using claim \eqref{equation:20} with $s_k =1\slash \alpha_k$ we have
	$$II= \int_{B_{\frac 1 {\alpha_k} \epsilon_k L_k}} e^{u^2_k}dx =
	2\pi \s_k^2 (\epsilon_kR^1_k)+o(1) = 8\pi n_2 +o(1).	$$
Hence we have 	$$\int_{B_{\alpha_k}}e^{w_k^2}d x= 8\pi (n^2+n'_2)+o(1).$$
Similarly we have $$\int_{B_{\alpha_k}}e^{w_k^3}dx= 8\pi (n^3+n'_3)+o(1).$$
	


\

\noindent
\textbf{Step 4:} We prove the conclusion $(1)$ by a induction argument.

\

Now we denote $\alpha_k\epsilon_kL_k$ as an new value of $R^1_k$ and repeat above processes. Since the  energies are uniformly bounded and after each step one of the local energy changes by a positive number, the process stops after finite times. We proved  $(1)$.

\

(2). Using the same argument in (1),
 the process stops after finite times, we can easily conclude that there exists $R^3_k\in [R^1_k, \tau_k\slash \epsilon_k]$ which is viewed as the initial value in the last step, such that $R^3_k=o(1)\frac{\tau_k}{\epsilon_k}$ and   $$\big(\sigma_1(\epsilon_kR^3_k), \sigma_2(\epsilon_kR^3_k), \sigma_3(\epsilon_kR^3_k)\big)=(4m_1,4m_2,4m_3)+o(1)\in\mathbf{V}+o(1),$$ for some $m_i\in\mathds{N}\cup \{0\},\ m_i\geq n_i,\ i=1,2,3$ and one of the following alternatives holds:
\begin{itemize}
  \item[(i)] $u^1_k,\ u^2_k,\ u^3_k$ have fast decay in $B_{\tau_k}\setminus B_{\epsilon_kR^3_k}$;
  \item[(ii)] There exist $i\in\{1,2,3\}$ and $L_k\in [R^3_k, \tau_k/\epsilon_k]$ such that $L_k=O(1)\tau_k/\epsilon_k$ and $u^i_k$ has slow decay on $\partial B_{\epsilon_kL_k}$.
\end{itemize}

By Lemma \ref{lem:03}, it is easy to see that  case (i) leads to  the first conclusion of $(2)$ immediately. If $(i)$ does not hold, then there exist $i\in\{1,2,3\}$ and $L_k\in [R^3_k,\tau_k/\epsilon_k]$, such that $u^i_k$ has slow decay on $\partial B_{\epsilon_kL_k}$.
By the choice of $R^3_k$, we get $L_k=O(1)\tau_k/\epsilon_k$. Otherwise, $R^3_k$ is not in the last step.  Next, we just need to show the case $(ii)$ yields the second conclusion of $(2)$. By the above arguments, we know there exists one component $i$ (still assume $i=1$) such that $$\sigma^1_k(\epsilon_kL_k)-\sigma^1_k(\epsilon_kR^3_k)=o(1)$$ and $$(w^2_k(x),w^3_k(x))=\big(u^2_k(\epsilon_kL_kx)+2\log \epsilon_kL_k, u^3_k(\epsilon_kL_kx)+2\log \epsilon_kL_k\big)\to (w^2,w^3),\ \ in\ \ C^2_{loc}(B_1\setminus \{0\}).$$ Moreover, there exists $\alpha_k\nearrow +\infty$ such that $$\sigma^i_k(\frac{1}{\alpha_k}\epsilon_kL_k)-\sigma^i_k(\epsilon_kR^3_k)=o(1)$$ and $$\int_{B_1\setminus B_{\frac{1}{\alpha_k}}}e^{w^i_k}dx-\int_{B_1\setminus B_{\frac{1}{\alpha_k}}}e^{w^i}dx=o(1),\ \ i=2,3.$$

Thus, for $i=2,3$, since $R^3_k=o(1)L_k$, we have
\begin{align*}
  \sigma^i_k(\epsilon_kR^3_k)\leq\sigma^i_k(s\epsilon_kL_k)&\leq \sigma^i_k(\frac{1}{\alpha_k}\epsilon_kL_k)+ \sigma^i_k(s\epsilon_kL_k)-\sigma^i_k(\frac{1}{\alpha_k}\epsilon_kL_k)\\
  &=\sigma^i_k(\epsilon_kR^3_k)+o(1)+2\pi \int_{B_s\setminus B_{\frac{1}{\alpha_k}}}e^{w^i}dx,
\end{align*}
which immediately implies
\begin{align*}
\lim_{s\to 0}\lim_{k\to\infty}\big(\sigma^1_k(s\epsilon_kL_k), \sigma^2_k(s\epsilon_kL_k), \sigma^3_k(s\epsilon_kL_k)\big)=\lim_{k\to\infty}\big(\sigma^1_k(\epsilon_kR^3_k), \sigma^2_k(\epsilon_kR^3_k), \sigma^3_k(\epsilon_kR^3_k)\big)\in\mathbf{V}.\end{align*} Then $$\lim_{s\to 0}\lim_{k\to\infty}\big(\sigma^1_k(s\tau_k), \sigma^2_k(s\tau_k), \sigma^3_k(s\tau_k)\big)\in\mathbf{V},$$ since $L_k=O(1)\tau_k/\epsilon_k$.

\

(3)  We prove  statement  $(3)$.

\

The first part follows immediately from the fact that there exists $i\in\{1,2,3\}$ such that $$\frac{d}{dr} \overline{u}^i_k(r)\leq \frac{-4+o(1)}{r},\ \ r=\epsilon_kR_k^3,$$ where $R^3_k$ is as in the proof of (2),  which is the initial value in the last
iteration. 

\

If  $u^i_k$ has fast decay on $\partial B_{\tau_k}$, then $u^i_k$ has fast decay in $B_{\tau_k}\setminus B_{\epsilon_kR^3_k}$ where $R^3_k$ is defined in the proof of (2).  Then
\begin{align*}
 \lim_{k\to\infty} \sigma^i_k(\tau_k)=\lim_{s\to 0}\lim_{k\to\infty}\sigma^i_k(s\tau_k)+ \lim_{s\to 0}\lim_{k\to\infty}\big(\sigma^i_k(\tau_k)-\sigma^i_k(s\tau_k)\big)=4n,
\end{align*} where the last equality follows from conclusion $(2)$ and the fact that
\begin{align*}
  \lim_{s\to 0}\lim_{k\to\infty}\big(\sigma^i_k(\tau_k)-\sigma^i_k(s\tau_k)\big)&=\lim_{s\to 0}\lim_{k\to\infty}\int_{B_{\tau_k}\setminus B_{s\tau_k}}e^{u^i_k}dx\\
  &=\lim_{s\to 0}\lim_{k\to\infty}e^{-N_k}\int_{B_{\tau_k}\setminus B_{s\tau_k}}\frac{1}{|x|^2}dx=0.
\end{align*}

We finished the proof of this proposition.

\end{proof}

\section{Proof of main theorem \ref{thm:main}}\label{sec:main-theorem}

\

In this section, we will prove our main  Theorem \ref{thm:main}. We first recall the definition of the Group given by \cite{Lin-Wei-Zhang} which is very useful to differentiate the bubble areas.

\begin{defn}\label{def}
Let $G:=\{p^1_k,...,p^m_k\}$ be a subset of $\Sigma_k$ with at least two points. $G$ is called a {\it group} if
\begin{itemize}
  \item[(1)] $\dist(p^i_k,p^j_k)\sim  \dist(p^s_k,p^t_k)$ for any points  $p^i_k,p^j_k,p^s_k,p^t_k\in G$ with $i\neq j,\ s\neq t$.
  \item[(2)] For any $p^i_k,p^j_k\in G$ with $i\neq j$ and any $p_k\in \Sigma_k\setminus G$, there holds $$\frac{\dist(p^i_k,p^j_k)}{\dist(p^i_k,p_k)}\to 0.$$
\end{itemize}
Here two sequences of positive numbers $\{a_k\}$ and $ \{b_k\}$  satisfying  $\{a_k\} \sim \{b_k\}$ means that $C^{-1} \le a_k/b_k \le C$, $\forall k$, for some constant $C>0$.
\end{defn}

\

As above by a standard translation argument we may assume $0\in\Sigma_k$. Let $G_1=\{0, x^1_k,...,x^m_k\}$ be the group containing $0$. Denote $$\varepsilon_k^1:=\frac{1}{2}\dist(G_1,\Sigma_k\setminus G_1).$$ Recall  that $\tau_k=\frac{1}{2}\dist(0,\Sigma_k\setminus \{0\})$.

\begin{prop}\label{prop:02}
Let $(u^1_k,\ u^2_k,\ u^3_k)$ be the solution of \eqref{equ:03}-\eqref{equ:06}. Then we have
\begin{itemize}
\item[(1)] For any $s_k\to 0$ with ${s_k}/{\tau_k}\to +\infty$ and $s_k=o(1)\varepsilon^1_k$, suppose $u^1_k,\ u^2_k,\ u^3_k$ have fast decay on $\partial B_{s_k}$. Then $$\big(\sigma^1_k(s_k), \sigma^2_k(s_k), \sigma^3_k(s_k)\big)\in\mathbf{V}+o(1).$$
  \item[(2)] On $\partial B_{\varepsilon_k^1}$, either $u^1_k,\ u^2_k,\ u^3_k$ have fast decay on $ B_{\varepsilon_k^1}$ and $$\big(\sigma^1_k(\varepsilon_k^1), \sigma^2_k(\varepsilon_k^1), \sigma^3_k(\varepsilon_k^1)\big)\in\mathbf{V}+o(1),$$ or there exists one component $u^i_k$ has slow decay on $\partial B_{\varepsilon_k^1}$ and $$\lim_{s\to 0}\lim_{k\to\infty}\big(\sigma^1_k(s\varepsilon_k^1), \sigma^2_k(s\varepsilon_k^1), \sigma^3_k(s\varepsilon_k^1)\big)\in\mathbf{V}.$$
  \item[(3)] There exists at least one component $u^i_k$ such that  $u^i_k$ has fast decay on $\partial B_{\varepsilon_k^1}$.
   Moreover, if $u^i_k$ has fast decay on $\partial B_{\varepsilon_k^1}$, then $\sigma^i_k(\varepsilon_k^1)=4n+o(1)$ for some $n\in \mathds{N}$.
\end{itemize}
\end{prop}
\begin{proof}
   Denote $\tau^j_k=\frac{1}{2}\dist(x^j_k,\Sigma_k\setminus \{x_k^j\})$, $j=1,...,m$, $x^j_k\in G_1$.   Then by Definition \ref{def},  there  exists a constant $N>0$ independent of $k$ such that
    $$\tau_k=o(1)\varepsilon_k^1,\ \ G_1\subset B_{N\tau_k}(0),\ \ \frac{1}{N}\leq \frac{\tau^j_k}{\tau_k}\leq N,\ \  \cup_{j=1}^mB_{\tau_k^j}(x^j_k)\subset B_{N\tau_k}(0),\ \ j=1,...,m,\ \forall k.$$
  Let $$v^i_k(x)=u^i_k(\tau_kx)+2\log\tau_k,\ |x|\leq\frac{1}{\tau_k},\ \ i=1,2,3.$$ By Theorem \ref{thm:01}, there exists at least one component $i\in\{1,2,3\}$ such that $v^i_k$ has fast decay on $\partial B_1$. 
  	Without loss of generality, we assume $i=1$. Then $$\sigma^1_k(\tau_k)=4n^0_1+o(1),\  \ \lim_{s\to 0}\lim_{k\to\infty}\big(\sigma^1_k(s\tau_k),\sigma^2_k(s\tau_k),\sigma^3_k(s\tau_k) \big)=(4n^0_1,4n^0_2,4n^0_3)\in \mathds{N}.$$
  Let $y^j_k=\frac{x^j_k}{\tau_k}$, $j=1,...,m$. Then $|y^j_k|\leq N$. Since  $\frac{1}{N}\leq \frac{\tau^j_k}{\tau_k}\leq N$, passing to a subsequence we assume $y^j_k\to y^j$ and $\frac{\tau^j_k}{\tau_k}\to r_j$. By Lemma \ref{lem:02}, we know that $v^1_k$ has fast decay on $\partial B_{r_j}(y^j)$. By Theorem \ref{thm:01}, we get $$\sigma (r_j,y^j;v^1_k)=4n^j_1+o(1)$$ and $$ \lim_{s\to 0}\lim_{k\to\infty}\big(\sigma (sr_j,y^j;v^1_k),\sigma (sr_j,y^j;v^2_k),\sigma (sr_j,y^j;v^3_k) \big)=(4n^j_1,4n^j_2,4n^j_3)\in \mathbf{V}.$$ Moreover, by the proof of Proposition \ref{prop:01} (see Step 5), for each $j=0,1,...,m$ (here, when $j=0$, $x^j_k=0,\ \tau^j_k=\tau_k,\ y^0=0,\ r_0=1$),  there exists $s_k^j=o(1)r_j$ (which can be viewed as the initial data in the last step of the iteration near $y^j$), such that $v^1_k,\ v^2_k,\ v^3_k$ have fast decay on $\partial B_{s_k^j}(y^j)$ and
   \begin{equation}\label{equation:24}
   \big(\sigma (s^j_k,y^j;v^1_k),\sigma (s_k^j,y^j;v^1_k),\sigma (s^j_k,y^j;v^1_k) \big)=(4n^j_1,4n^j_2,4n^j_3)+o(1)\in \mathbf{V}+o(1).
   \end{equation}

  By Lemma \ref{lem:02} and Lemma \ref{lem:05}, we obtain that: $(i)$ $v^1_k$ has fast decay in $B_{N+1}\setminus \cup_{j=0}^m B_{r_j}(y^j)$ and $$\sigma (N+1, 0; v^1_k)=\sum_{j=0}^m4n_1^j+o(1);$$ $(ii)$ There exists $R_k\to +\infty$ such that $R_k=o(1)\frac{\varepsilon^1_k}{\tau_k}$ and $v^1_k$ has fast decay in $B_{R_k}\setminus B_{N+1}$ and
  \begin{equation}\label{equ:13}
    \sigma (R_k,0;v^1_k)-\sigma (N+1,0;v^1_k)=o(1).
  \end{equation}

Now we have to consider following two cases:

\

\noindent\textbf{Case 1.}  $v^2_k,v^3_k$ have fast decay on $\partial B_1$.

\

Then by Lemma \ref{lem:02} and Proposition \ref{prop:01}, we get that $v^2_k,v^3_k$ have fast decay in $B_{N+1}\setminus \cup_{j=0}^m B_{r_j}(y^j)$ and
\begin{align*}
 \sigma(r_j,y^j;v^2_k)=4n^j_2+o(1),\ \ \sigma(r_j,y^j;v^3_k)=4n^j_3+o(1),\ \ j=0,...,m.
\end{align*}
Thus, $$\sigma(N+1, 0; v^i_k)=\sum_{j=0}^m4n_i^j+o(1),\ \ i=1,2,3.$$  Now, we know that  $u^1_k,u^2_k,u^3_k$ have fast decay on $\partial B_{(N+1)\tau_k}$ and $$\big(\sigma_1((N+1)\tau_k), \sigma_2((N+1)\tau_k), \sigma_3((N+1)\tau_k)\big)=\big(\sum_{j=0}^m4n_1^j,\sum_{j=0}^m4n_2^j,\sum_{j=0}^m4n_3^j\big)+o(1)\in\mathbf{V}+o(1).$$ Then the conclusions of the proposition follow from Proposition \ref{prop:01}.

\

\noindent\textbf{Case 2.}  $v^2_k$ or $v^3_k$ has slow decay on $\partial B_1$.

\

By Lemma \ref{lem:01}, we know $$\max_{i=1,2,3}u^i_k(x)+2\log \dist(x,\Sigma_k)\leq C,\ \ \forall \ x\in B_1.$$ Then it is easy to see that  $$\max_{i=1,2,3}v^i_k(x)\leq C(\delta),\ \ v^1_k(x)\to -\infty, \ \ \forall \ x\in B_{R_k}\setminus \cup_{j=0}^mB_{\delta}(y^j),$$ for any fixed $\delta>0$.
Since  $v^2_k$ or $v^3_k$ has slow decay on $\partial B_1$ and  $v^i_k,\ i=1,2,3$ also satisfy system  \eqref{equ:03}-\eqref{equ:05}, by the standard elliptic theory and a similar argument as in {Step 3} in Proposition \ref{prop:01}, we have  $v^1_k(x)\to -\infty,$ locally uniformly in $\R^2\setminus \{0,y^1,...,y^m\}$ and one of the following alternatives holds:
\begin{itemize}
  \item[(a)] $v^2_k\to -\infty$  locally uniformly in $\R^2\setminus \{0,y^1,...,y^m\}$ and $v^3_k\to v^3$ in $C^2_{loc}(\R^2\setminus \{0,y^1,...,y^m\})$  where $w^3$ satisfies $$-\Delta v^3=-\sum_{j=0}^m4\pi (n^j_1+n^j_2-2n^j_3)\delta_{y^j}+e^{v^3}\ \ in\ \R^2,$$ where  $\delta_{y^j}$ is the Dirac measure at the point $y^j$.
  \item[(b)] $v^3_k\to -\infty$  locally uniformly in $\R^2\setminus \{0,y^1,...,y^m\}$ and $v^2_k\to v^2$ in $C^2_{loc}(\R^2\setminus \{0,y^1,...,y^m\})$  where $w^2$ satisfies $$-\Delta v^2=-\sum_{j=0}^m8\pi (n^j_3-n^j_2)\delta_{y^j}+e^{v^2}\ \ in\ \R^2.$$
  \item[(c)] $(v^2_k,v^3_k)\to (v^2, v^3)$ in $C^2_{loc}(\R^2\setminus \{0,y^1,...,y^m\})$  where $(v^2,v^3)$ satisfies
  \begin{align*}
    \begin{cases}
    -\Delta v^2&=-\sum_{j=0}^m8\pi (n^j_3-n^j_2)\delta_{y^j}+e^{v^2}-e^{v^3},\\
      -\Delta v^3&=-\sum_{j=0}^m4\pi (n^j_1+n^j_2-2n^j_3)\delta_{y^j}-\frac{1}{2}e^{v^2}+e^{v^3},
    \end{cases}in\ \R^2.
  \end{align*}
\end{itemize}

We now claim that the constants in $(c)$ satisfy $$2(n^j_3-n^j_2)>-1,\ \ n^j_1+n^j_2-2n^j_3>-1,\ \ j=0,1,...,m.$$ First, it is easy to see that $$2(n^j_3-n^j_2)>-1, \ \ j=0,1,...,m.$$  Otherwise there exits $j\in \{0,1,...,m\}$ such that $2(n^j_3-n^j_2)\leq -1$. Then $2(n^j_3-n^j_2)\leq -2$  and $$r\frac{d}{dr}\overline{v}^2_k(r)=4(n^j_3-n^j_2)+o(1)\leq -4+o(1),\ \ r=s_k^j,$$  where $s_k^j=o(1)r_j$ is defined in \eqref{equation:24}. Then by the proof of Proposition \ref{prop:01}, we get that $v^2_k$ has fast decay in $B_{r_j}(x^j)\setminus B_{s^j_k}(x^j)$ which implies $v^2_k\to -\infty$  locally uniformly in $\R^2\setminus \{0,y^1,...,y^m\}$. This is a contradiction.

Now we show $$n^j_1+n^j_2-2n^j_3>-1,\ \ j=0,1,...,m.$$ In fact, if not, then there exists $j\in \{0,1,...,m\}$ such that $n^j_1+n^j_2-2n^j_3\leq -1$. Then there must be $$n^j_1+n^j_2-2n^j_3= -1.$$ Otherwise we have $n^j_1+n^j_2-2n^j_3\leq -2$ and $$r\frac{d}{dr}\overline{v}^3_k(r)=2(n^j_1+n^j_2-2n^j_3)+o(1)\leq -4+o(1),\ \ r=s_k^j,$$  where $s_k^j=o(1)r_j$ is defined in \eqref{equation:24}. Similarly, we know that this is also a contradiction.

By \eqref{equation:12}, \eqref{equation:13} and  Lemma \ref{lem:04}, we know that among the following constants $$n_3^j-n_1^j,\ \  n_3^j-n_2^j,\ \  \frac{1}{2}(n_1^j+n_2^j)-n_3^j,$$ there exist one component less than $-1$ and one component bigger than $\frac{1}{4}$. Since $n^j_1+n^j_2-2n^j_3= -1$, we know $n_3^j-n_1^j\leq -1,\ n_3^j-n_2^j\geq\frac{1}{4}$ or $n_3^j-n_2^j\leq -1,\ n_3^j-n_1^j\geq\frac{1}{4}$. Without loss of generality, we assume the first case holds. Then $$n_3^j-n_1^j\leq -1,\ n_3^j-n_2^j\geq 1.$$ By the Pohozaev equality \ref{lem:Pohozaev} $$(n_3^j-n_1^j)^2+(n_3^j-n_2^j)^2=(n_1^j+n_2^j+2n^j_3)$$ and the fact $n^j_1+n^j_2-2n^j_3= -1$, we get $$(n_3^j-n_1^j)(n_3^j-n_1^j-1)=2n^j_3-1.$$ This is a contradiction since the left side is even while the right side is odd.

Similarly, one can prove that the constants in $(a)$ and $(b)$ satisfy $$2(n_3^j-n_2^j)>-1,\ \ n^j_1+n^j_2-2n^j_3> -1.$$

By the classification results in \cite{Lin-Yang-Zhong}, for above three cases,  we have $$\big(\int_{\R^2}e^{v^2}dx,\int_{\R^2}e^{v^3}dx\big)=(8\pi n'_2, 8\pi n'_3),\ \ n'_1, \ n'_2\in\mathds{N}\cup\{0\}.$$ By the argument as in
{Step 3} in proposition \ref{prop:01}, we can find a sequence of numbers $R^1_k\to +\infty$  such that $R^1_k=o(1)R_k$ where $R_k$ is given in \eqref{equ:13}, $v^1_k,v^2_k,v^3_k$ have fast decay on $\partial B_{R^1_k}$ and
\begin{align*}
\sigma^1_k(R^1_k\tau_k)=\sum_{j=0}^m4n_1^j+o(1), \ \sigma^2_k(R^1_k\tau_k)=\sum_{j=0}^m4n_2^j+4n_2'+o(1), \ \sigma^3_k(R^1_k\tau_k)=\sum_{j=0}^m4n_3^j+4n_3'+o(1)
\end{align*} and $$\big(\sigma^1_k(R^1_k\tau_k),\sigma^2_k(R^1_k\tau_k),\sigma^3_k(R^1_k\tau_k)\big)\in \mathbf{V}+o(1).$$ Then the conclusions of the proposition follow from Proposition \ref{prop:01}.

\end{proof}

\

\noindent\textit{Proof of Theorem \ref{thm:main}.} Since now we have proved Theorem \ref{thm:01} and Proposition \ref{prop:02}, then by applying a global Pohozaev identity, the conclusion of  Theorem \ref{thm:main} follows from more or less  standard argument
 which has already used in several papers, see for example \cite{Jevnikar-Wei-Yang}.    We omit the details here.

\

\section{ Affine Toda system}\label{sec_last}

Another closely related completely integrable system is the following
\begin{equation}\begin{array}{rcl}\label{eq_a1}
-\Delta \theta &=& e^{2\theta}-e^{-\theta} \cosh 3\phi\\
-\Delta \phi &=& e^{-\theta} \sinh 3\phi
\end{array}
\end{equation}
See \cite{FG}. When $\phi=0$, it reduces to a scalar equation
\begin{eqnarray}\label{eq_a2}
-\Delta \theta &=& e^{2\theta}-e^{-\theta},
\end{eqnarray}
which is usually called Tzitz\'{e}ica equation \cite{Tzitzeca-2}.
Its  blow-up analysis was studied recently by Jevniker-Yang \cite{Jevnikar-Yang}. Let
\begin{eqnarray*}
u_1&=& -\theta+3\phi \\
u_2 &=& -\theta-3\phi \\
u_3 & = & 2\theta.
\end{eqnarray*}
It is clear  to see that the above system is equivalent to
\begin{equation}\begin{array}{rcl}\label{eq_a3}
-\Delta u_1  &=& 2 e^{u_1}-e^{u_2}-e^{u_3}  \\
-\Delta u_2  &=&-e^{u_1}+2e^{u_2}-e^{u_3}  \\
-\Delta u_3  &=&-e^{u_1}-e^{u_2}+2e^{u_3} \\
u_1+u_2+u_3 &=& 0.
\end{array}
\end{equation}
It is usually  the affine Toda system for  $SU(4)$. Our method works also for this system. In fact, the argument would become slightly easy, because due to  the symmetry of the system, the blow-up has only two different cases,
either one component or two components blow up. For the former one obtains a solution to the Liouville equation as a limit, while for the latter case a solution to the  SU(3) Toda system. Both are the so-called  open Toda systems mentioned in the Introduction. In order to give a similar form as in Theorem \ref{thm:main} we consider its rescaled system
\begin{equation}\begin{array}{rcl}\label{eq_a4}
-\Delta u_1  &=&  e^{u_1}- \frac 12 e^{u_2}- \frac 12 e^{u_3}  \\
-\Delta u_2  &=&- \frac 12 e^{u_1}+e^{u_2}- \frac 12 e^{u_3}  \\
-\Delta u_3  &=&-\frac 12e ^{u_1}- \frac 12 e^{u_2}+e^{u_3} \\
u_1+u_2+u_3 &=& 0.
\end{array}
\end{equation}

\begin{thm}
	Theorem \ref{thm:main} holds for system \eqref{eq_a4}. Namely, under similar assumptions,
	the local mass $\s_i$ ($i=1,2,3$) are a multiple of $4$, which satisfy the corresponding   Pohozaev identity
	\[
	(\sigma_1-\sigma_2)^2 + (\sigma_2-\sigma_3)^2+ (\sigma_3-\sigma_1)^2=12 (\sigma_1
	+\sigma_2+\sigma_2).\]
\end{thm}

It is  interesting to see that system \eqref{eq:03}-\eqref{eq:06}  can be embedded  in  the affine Toda system for $SU(N+1)$ with $N=5$, i.e., system \eqref{eq_0}
with an $N\times N$ coefficient matrix $A$

  \[A:=
  \left(\begin{array}{rrrrrr}
  2 & -1 & 0 & \cdots & 0 & -1\cr
  -1& 2 &-1& 0 & \cdots & 0\cr
  0&-1&2&-1& \cdots & 0\cr
  \cdots&\cdots&\cdots&\cdots&\cdots&\cdots\cr
  0 &\cdots&\cdots& -1&2&-1\cr
  -1& 0 &\cdots & 0 &-1&2\cr\end{array}\right).\nonumber
  \]
Hence it would be interesting to consider the general case.
It is easy to observe that all blow-up limits of an affine  Toda system are solutions to various open Toda systems. However, for the affine Toda system for $SU(N+1)$ with $N\ge 4$, the argument given here is not yet enough, since
it is more than $3$ unknown functions.
We will consider this problem later.

\section{Appendix}\label{sec:appendix}

In this Appendix we recall two classification results.

\begin{thm}[Classification theorem]\label{thm:classification-1}
Suppose $(u,v)$ is a solution of
\begin{align}\label{equ:09}
\begin{cases}
  -\Delta u&=e^u-e^v, \\
  -\Delta v&=-\frac{1}{2}e^u+e^v,
\end{cases}\end{align}in $\R^2$, with $\int_{\R^2}e^udx+\int_{\R^2}e^vdx<\infty.$ Then we have $$\int_{\R^2}e^udx=32\pi,\ \ \int_{\R^2}e^vdx=24\pi $$ and $$u(x)=-4\log |x|+O(1),\ \ v(x)=-4\log |x|+O(1)\ as \ |x|\to\infty.$$
\end{thm}
\begin{proof}
Let $\tilde{u}(x)=u(\sqrt{2}x),\ \tilde{v}(x)=v(\sqrt{2}x)$. Then
\begin{align*}
  -\Delta\tilde{u} &=2e^{\tilde{u}}-2e^{\tilde{v}}, \\
  -\Delta \tilde{v}&=-e^{\tilde{u}}+2e^{\tilde{v}},
\end{align*} in $\R^2$, with $\int_{\R^2}e^{\tilde{u}}dx+\int_{\R^2}e^{\tilde{v}}dx<\infty.$
One can easily embed the above system into the open Toda system for $SU(4)$  with $\tilde v= u_1=u_3$ and $\tilde u=u_2$. Then
 the classification theorem proved in \cite{JW0} implies
$$\int_{\R^2}e^{\tilde{u}}dx=16\pi,\ \ \int_{\R^2}e^{\tilde{v}}dx=12\pi $$ and $$\tilde{u}(x)=-4\log |x|+O(1),\ \ \tilde{v}(x)=-4\log |x|+O(1)\ as \ |x|\to\infty.$$ Then the conclusion of the theorem follows immediately.
\end{proof}

\

\begin{thm}[Classification theorem]\label{thm:classification-2}
Suppose $(u,v)$ is a solution of
\begin{align}
\begin{cases}
  -\Delta u&=e^u-e^v-\sum_{j=1}^m4\pi \alpha_j\delta_{p_j}, \\
  -\Delta v&=-\frac{1}{2}e^u+e^v-\sum_{j=1}^m4\pi \beta_j\delta_{p_j},
\end{cases}\end{align}in $\R^2$, with $\int_{\R^2}e^udx+\int_{\R^2}e^vdx<\infty$, where  $\alpha_j,\ \beta_j\in \mathds{N}\cup\{0\},\ j=1,...,m$. Then we have $$\frac{1}{2\pi}\int_{\R^2}e^udx=4n_1,\ \ \frac{1}{2\pi}\int_{\R^2}e^vdx=4n_2 ,$$ where $n_1,\ n_2\in \mathds{N}$.
\end{thm}
\begin{proof} As in Theorem \ref{thm:classification-1}, let $\tilde{u}(x)=u(\sqrt{2}x),\ \tilde{v}(x)=v(\sqrt{2}x)$.
It is clear that we have
\begin{align*}
\begin{cases}
  -\Delta \tilde{u} &=2e^{\tilde{u}}-2e^{\tilde{v}}-\sum_{j=1}^m8\pi \alpha_j\delta_{p_j}, \\
  -\Delta \tilde{v}&=-e^{\tilde{u}}+2e^{\tilde{v}}-\sum_{j=1}^m8\pi \beta_j\delta_{p_j},
\end{cases}\end{align*} in $\R^2$, with $\int_{\R^2}e^{\tilde{u}}dx+\int_{\R^2}e^{\tilde{v}}dx<\infty.$ By the classification result Corollary 2.3 in \cite{Lin-Yang-Zhong}, we have $$\frac{1}{2\pi}\int_{\R^2}e^{\tilde{u}}dx=2n_1,\ \ \frac{1}{2\pi}\int_{\R^2}e^{\tilde{u}}dx=2n_2 ,$$ where $n_1,\ n_2\in \mathds{N}$. Thus  $$\frac{1}{2\pi}\int_{\R^2}e^udx=4n_1,\ \ \frac{1}{2\pi}\int_{\R^2}e^vdx=4n_2 .$$
\end{proof}

\


\providecommand{\bysame}{\leavevmode\hbox to3em{\hrulefill}\thinspace}
\providecommand{\MR}{\relax\ifhmode\unskip\space\fi MR }
\providecommand{\MRhref}[2]{%
  \href{http://www.ams.org/mathscinet-getitem?mr=#1}{#2}
}
\providecommand{\href}[2]{#2}

\end{document}